\documentclass[11pt]{amsart}
\usepackage{array, amsmath, amscd, amssymb, amsthm, mathrsfs, mathtools}
\usepackage{hyperref, xy}

\newcommand{\cYn}{\cY^{[n]}}

\newcommand{\Hilb}[2]{{#1}^{[#2]}}

\newcommand{\cEn}{\cE^{[n]}}

\newcommand{\Gott}{G\"{o}ttsche}

\newcommand{\Pand}{Pandharipande}

\newcommand{\fourtop}{L^2, LK, c_1(S)^2, c_2(S)}

\newcommand{\sch}{\mathbf{Sch}_k}
\newcommand{\sm}{\mathbf{Sm}_k}

\newcommand{\codim}{\text{codim}}
\newcommand{\Spec}{\text{Spec}\,}
\newcommand{\ri}{r_1, \ldots, r_k}

\theoremstyle{plain}

\newtheorem{prop}{Proposition}[section]
\newtheorem{theo}[prop]{Theorem}
\newtheorem{lemm}[prop]{Lemma}
\newtheorem{exam}[prop]{Example}
\newtheorem{assu}[prop]{Assumption}
\newtheorem{coro}[prop]{Corollary}

\theoremstyle{definition}
\newtheorem{defn}{Definition}[section]

\newtheorem{eg}{Example}[section]

\theoremstyle{remark}
\newtheorem*{rem}{Remark}

\newcommand{\CC}{\mathbb{C}}

\newcommand{\FF}{\mathbb{F}}
\newcommand{\LL}{\mathbb{L}}

\newcommand{\NN}{\mathbb{N}}
\newcommand{\PP}{\mathbb{P}}
\newcommand{\QQ}{\mathbb{Q}}

\newcommand{\ZZ}{\mathbb{Z}}

\newcommand{\bM}{\mathbf{M}}

\def\cA{{\mathcal A}}

\def\cC{{\mathcal C}}
\def\cD{{\mathcal D}}
\def\cE{{\mathcal E}}

\def\cM{{\mathcal M}}
\def\cN{{\mathcal N}}
\def\cO{{\mathcal O}}
\def\cP{{\mathcal P}}
\def\cQ{{\mathcal Q}}
\def\cR{{\mathcal R}}

\def\cV{{\mathcal V}}

\def\cY{{\mathcal Y}}

\def\fm{\mathfrak{m}}

\let\oldTitle\title
\renewcommand{\title}[1]{\newcommand{\myTitle}{#1}\oldTitle{#1}} 
\title{Enumeration of singular varieties with tangency conditions}

\input xy
\xyoption {all}
\numberwithin{equation}{section}
\begin{document}

\author{Yu-jong Tzeng}
\email{yujong@gmail.com}

\thanks{The author is supported by NSF grant DMS-1503621.}

\begin{abstract}  
We construct the algebraic cobordism theory of  bundles and  divisors on smooth varieties. 
It has a simple basis (over $\QQ$) from projective spaces and its rank is equal to the number of Chern invariants. 
As an application we study the number of singular subvarieties satisfying given tangent conditions with a fixed smooth divisor, where the subvariety is the zero locus of a section of a vector bundle. 
We prove that the generating series  gives a homomorphism from the algebraic cobordism group to the power series ring. 
This implies that the number of singular subvarieties with tangency conditions is governed by universal polynomials of Chern numbers, when the vector bundle is sufficiently ample. 
This result combines and generalizes the Caporaso-Harris and Vakil's recursive formulae, \Gott's conjecture,  De Jonqui\`ere's Formula and relative node polynomials from tropical geometry.
  
\end{abstract}
\maketitle

\section{Introduction}
\subsection{History of the problem}
The number of (possibly reduced)  $\delta$-nodal degree $d$  curves  passing through general $\frac{d(d+3)}{2}-\delta$ points on the projective plane is classically known as the Severi degree $N^{d,\delta}$.
They remained unknown for a long time  until Caporaso and Harris  \cite{CH} found a recursive formula to determine all of them.
This recursive formula contains not only the Severi degrees but also  the number of $\delta$-nodal curves satisfying tangency conditions at assigned points (indexed by $\alpha$) and unassigned points (indexed by $\beta$), which is called the Caporaso-Harris invariants $N^{d, \delta}(\alpha, \beta)$. 

Di Francesco and C. Itzykson \cite{FI} conjectured that for any fixed $\delta$, the Severi degrees $N^{d, \delta}$ are given by a polynomial in $d$ if  $d$  is large enough. 
This conjecture was proven by Fomin and Mikhalkin \cite{FM} in 2009 using  tropical geometry to reduce the problem to  counting certain combinatorial diagrams.   
Block \cite{BlockComputing} developed an algorithm to computed the polynomial and gave the result for $\delta \leq 14$.  Furthermore, Adila and Block \cite{ArdilaBlock} proved the polynomial property for Severi degrees holds for more general tropical surfaces, some of them even singular. 

For sufficiently ample line bundles $L$ on smooth surfaces $S$, \Gott\ \cite{Gott} conjectured that the number of $\delta$-nodal curves in the linear system $|L|$ is given by a universal polynomial of degree $\delta$ in Chern invariants $L^2$, $LK_S$, $c_1(S)^2$ and $c_2(S)$. The case of $\delta\leq 3 $ can be computed directly by standard intersection theory and was known in the 19th century. 
In 1994, Vainsencher  \cite{Va} proved and computed the universal polynomials for $\delta\leq 6$ and  later Kleiman and Piene \cite{KP} improved the result to $\delta\leq 8$. 
In 2000, A.K. Liu \cite{Liu1, Liu2} proposed a proof of existence of universal polynomials for all $\delta$ from the perspective of  sympletic geometry.
 
The first complete proof of \Gott's conjecture was found by the author \cite{Tz} by viewing the number of nodal curves as invariants of the algebraic cobordism group of surfaces and line bundles.
Shortly after, Kool, Shende and Thomas \cite{KST} also gave another beautiful proof. 
Gottsche's conjecture has been generalized to hypersurfaces with arbitrary analytic singularities \cite{LT, Renn}. 

In this paper, we will complete the picture by defining a natural generalization of Caporaso-Harris invariants to count singular varieties on any smooth algebraic varieties satisfying given tangency conditions with a smooth divisor. 
The singularities and tangent points are required to be isolated.
When ampleness conditions are satisfied, these invariants should ``satisfy \Gott's conjecture'', i.e. they are also given by universal polynomials and generating series has similar features. 
The reason behind the theme is algebraic cobordism theory and we will illustrate the connection between these two topics. 

\subsection{Tangency conditions and Non-induced singular type}\label{sect:non-induced}
On surfaces, if two curves do not have common components then their intersection types at smooth points can be completely classified by the intersection multiplicities. 
Therefore the tangent conditions of a plane curve with a fixed line can be recorded by two sequences of nonnegative integers  $\alpha=(\alpha_1, \alpha_2,\ldots)$ and  $\beta=(\beta_1, \beta_2,\ldots)$, which indicate intersection multiplicity $i$ at $\alpha_i$ assigned points and $\beta_i$ unassigned points. 
The  numbers of  $\delta$-nodal degree $d$  curves  having tangency type $(\alpha, \beta)$ with a fixed line are  Caporaso-Harris invariants $N^{d, \delta}(\alpha, \beta)$,

Caporaso and Harris' recursive formula \cite{CH} completely determines all Caporaso-Harris invariants, but it is hard to see other  structures from it. 
Inspired by Fomin and Mikhalkin's methods, Block \cite{Blockrel} showed that  $N^{d, \delta}(\alpha, \beta)$ are polynomials in the components of $\alpha$ and $\beta$ multiplied by some  natural coefficients if $\sum \beta_i \geq \delta$. They are called relative node polynomials and Block gave the formula when $\delta\leq 6$. 

In higher dimensions, the tangency conditions of a subvariety with a divisor are more complicated.
We only study the case when the intersection has expected dimension and all singular and tangent points are isolated.  
For example, our discussion does not include  a sphere touching a plane only at a point, or the parabolic cylinder $V(z=x^2)$ intersecting the $xy$-plane at a double line in $3$-space, because their intersection or tangent points does not have expected dimension. 
 
If the subvariety is the vanishing locus of sections of vector bundles and its intersection with the divisor  has isolated singularities, then the singularities must be isolated complete intersection singularities (ICIS). 
For this reason, the  tangency conditions imposed  in this article  are collections of ICIS. 
The tangency conditions are recorded by two finite collections of ICIS:  $\alpha$ (at assigned points) and $\beta$ (at unassigned points). 
Caporaso-Harris invariants can be recovered by imposing suitable ICIS. 

The subvariety may be tangent to the divisor at its singular points. 
In this case we should not view the subvariety as satisfying both the singularity and tangency conditions because the expected dimension is wrong. 
The following definitions  will help to avoid this ambiguity.  
\begin{defn} Let $X$ be a subvariety and $D$ be a smooth divisor of a variety $Y$. We say $X$ is tangent to $D$ at $p$ if the scheme-theoretic intersection $X\cap D$ is singular at $p$.
The \textit{tangency type} of $X$ with $D$ is the singularity type of $X\cap D$. 
\end{defn}

\begin{defn}  Let $X$ be a subvariety and $D$ be a smooth divisor of a variety $Y$.
 A \textit{non-induced singularity point} of $X$ is a singular point of $X$ which is not on $D$.   The \textit{non-induced singular type} of $X$ is the collection of singularity types of $X$ at its non-induced singular points.  
\end{defn}

If $X$ is singular at $p\in X\cap D$ and $X\cap D$ has codimension $1$ in $X$. Then a computation of tangent spaces shows $X\cap D$ must be singular at $p$. 
As a result,  $X$ is always tangent to $D$ at singular points of $X$, if the points lie on $D$.
Therefore a singular point of $X$ is either  non-induced or $X$ is tangent to $D$ at this point.

\subsection{Main results} 

If  $\delta$ is a  collection of ICIS,  denote the number of elements in $\delta$  by $|\delta|$.
$N(\delta)$ and $\tau(\delta)$ are invariants of $\delta$ and will be defined in Section \ref{sect:notations}.  
A vector bundle $E$ on $Y$ is \textit{$k$-very ample} if for every zero-dimensional scheme $Z$ of length $k+1$ in $Y$, the natural restriction map $H^0(E)\to H^0(E\otimes \cO_Z)$ is surjective. 

\begin{theo}\label{thm:univ}
	For any positive integers $n$, $r$, let  $\alpha$, $\beta$ denote collections of dimension $n-r-1$ ICIS and  $\delta$ denote a collection of dimension $n-r$ ICIS.  
	For fixed $\beta$ and $\delta$, there is a degree $|\beta|+|\delta|$ universal polynomial  $T_{\beta, \delta}$  satisfying the following property. 
	For any smooth variety $Y$ of dimension $n$, $\alpha$, an $(N(\alpha)+N(\beta)+N(\delta)+n-r+1)$-very ample vector bundle $E$ of rank $r$ and  a smooth divisor $D$  on $Y$,  the number of subvarieties in a general $(\tau(\delta)+(n-1)|\alpha|+\tau(\alpha)+\tau(\beta))$-dimensional  linear subspace  $\PP(V)\subset \PP^0(H^0(E))$ with non-induced singularity type $\delta$ and tangency type $(\alpha,\beta)$ with  $D$ is given by the polynomial $T_{\beta, \delta}$ in the Chern numbers of $Y$, $D$, and $E$ and the multiplicities of ICIS in $\alpha$. 

\end{theo}
\begin{rem} In the case of  hypersurfaces, if $X$ is defined by $f(x_1, x_2,\dots, x_n)$ and  $p$ is a given point on $D := (x_1=0)$ near $p$, then 
$X$ is tangent to $D$ at $p$ means $f(p)=\frac{\partial f}{\partial x_i}(p)=0$ for $i = 2,\ldots, n$.  
The conditions for $p$ to be a singular point of $X$ is $f(p)=\frac{\partial f}{\partial x_i}(p)=0$ for $i = 1,\ldots, n$.  
Since the codimension of the former is $n$ and the latter is $n+1$, 
in the generic case (which can be guaranteed if the line bundle is sufficiently ample), we can drop the ``non-induced'' assumption in the theorem. 
In other words, the subvarieties enumerated by Theorem \ref{thm:univ} generically is tangent to $D$ at smooth points of $X$. 

 \end{rem}
 
Since the degrees of universal polynomials are known, their coefficients can be computed by special cases. 
Consider the generating series 
$$ T_{\alpha}(Y,D,E) = \sum_{\beta, \delta} T_{\alpha, \beta, \delta}(\text{Chern numbers of } Y, D, E) y_{\beta}z_{\delta}.$$ 

The following theorem gives the structure of the generating series. 
It induces many relations between universal polynomials. 

\begin{theo}
	If $\{\Theta_1, \Theta_2,\ldots,  \Theta_m \}$ forms a basis of the finite-dimensional $\QQ$-vector space of graded degree $n$ polynomials in the Chern classes of 
		$$\{c_i(T_Y)\}_{i=0}^{n}, c_1(\cO(D)),\{c_i(E)\}_{i=0}^{r}.$$
	Then there exist  power series $A_j$ and $P_{\alpha_i}$ in  $\QQ[[y_{\beta}, z_{\delta}]]$ such that the generating series has the form
\begin{align*}T_{\alpha}(Y,D,E)
= \prod_{j=1}^m A_j^{\Theta_j(c_i(T_Y), c_1(D), c_i(E) )}\cdot \prod_{\alpha_i \in \alpha} P_{\alpha_i}. \end{align*}

 \end{theo}

\subsection{Outline of the paper}
This paper is divided into two parts. 
The first part discusses algebraic cobordism theory.
First we give a review of algebraic cobordism theory of bundles on varieties and  generalize it to arbitrary number of bundles  in Section \ref{sect:algbundles}.  
The algebraic cobordism theory of bundles and divisors on varieties is constructed in Section \ref{sect:algcobordism}, where we prove that there is a natural isomorphism between these two algebraic cobordism theories and the isomorphism gives us a natural basis. 
The second part is the enumeration of singular varieties  satisfying prescribed tangency conditions with a divisor. 
We begin with a brief review about isolated complete intersection singularities and  notations in Section \ref{sect:notations}. 
In Section \ref{sect:deg}, the number of singular varieties is expressed as an intersection number on Hilbert schemes of points and two degeneration formulae are derived.  
The degeneration formulae combined with the structure of algebraic cobordism group lead us to the proof of main theorems in Section \ref{sec:univ}. 
In the last section, we provide a list of current developments of the enumeration of singular varieties with tangency conditions. 

\subsection{Acknowledgments} I am indebted to  Dawei Chen and Jun Li for convincing me in the beginning that the construction in this paper is possible. 
I am grateful for Xiaowen Hu for pointing out an issue in the earlier version of the article and Rahul Pandharipande for explaining the details of \cite[Proposition 5]{LeeP}. 
I also wish to express my gratitude to Lawrence Ein,  Y.P. Lee, Jason Lo,  Ravi Vakil, Baosen Wu, and Ziyu Zhang for helpful discussions and encouragement.  
This project is supported by NSF grant DMS-1503621 and most of the work is done when the author is at University of Minnesota.


\section{Algebraic cobordism of bundles on varieties}\label{sect:algbundles}
\subsection{Algebraic cobordism ring of varieties}\label{var}


In this section, we will study the algebraic cobordism theory of arbitrary number of bundles on varieties.
First we give  an exposition of the algebraic cobordism theory of varieties in \cite{LP} and the algebraic cobordism theory of a pair of variety and bundle in \cite{LeeP}. 
Then we show how the approach in \cite{LeeP} can be slightly generalized to the case of arbitrary number of bundles. 

Let $k$ be a field of characteristic $0$, $\mathbf{Sch}_k$ be the category of separated schemes of finite type over $k$, and  $\mathbf{Sm}_k$ be the full subcategory of smooth quasi-projective $k$-schemes. 

For $X\in \sch$, let $\cM(X)$ be the set of  the isomorphism classes over $X$ of projective morphisms $f: M \to X$ with $M\in \sm$. The set $\cM(X)$ is graded by the dimension of $M$ and it has the structure of a monoid under  disjoint of domains. Let $\cM_*(X)^+$ be the graded group completion of $\cM(X)$ and $[f:M\to X]$ be the element in  $\cM_*(X)^+$ for $f$. 

Consider projective morphisms 
$$\pi: Y\to X\times \PP^1$$ and composition $$\pi_2= p_2\circ \pi: Y\to \PP^1\,$$ where $p_2:X\times \PP^1\to \PP^1$ is the projection to the second factor, $Y\in \sm$ and $X\in\sch$.
We say $\pi_2$ is a \textit{double point degeneration}\footnote{The point $0$ can be replaced by any regular value of $\pi$.} over $\infty\in \PP^1$ if   $Y_0:=\pi_2^{-1}(0)$ is a smooth fiber  over $0$, and $\pi_2^{-1}(\infty)$ can be written as $Y_1\cup Y_2$ where $Y_1$ and $Y_2$ are smooth Cartier divisors of $Y$ intersecting transversally along a smooth divisor $B=Y_1\cap Y_2$. 
Let $N_{Y_1/B}$ and $N_{Y_2/B}$ denote the normal bundles of $B$ in $Y_1$ and $Y_2$ respectively and  $$Y_3:=\PP(\cO_B\otimes N_{Y_1/B})\cong \PP( N_{Y_2/B}\otimes\cO_B).$$ be a  $\PP^1$ bundle over $B$. 
Then the \textit{double point relation} over $X$ defined by $\pi$ is  
\begin{align*} 
[Y_0\to X]-[Y_1\to X]-[Y_2\to X]+[Y_3\to X]\,. \end{align*}
Denote the subgroup of $\cM_*(X)^+$ generated by all double point relations over $X$ by $\cR_*(X)$. 
The \textit{algebraic cobordism theory} $\omega_*$ is defined by $$\omega_*(X):=\cM_*(X)^+/\cR_*(X).$$ Since the double point relation is homogeneous, $\omega_*(X)$ is graded by the dimension of the domain. 

The first algebraic cobordism theory $\Omega_*$ was constructed by Morel and Levine \cite{LM} to be the universal oriented Borel-Moore homology theory of schemes. 
A key theorem of \cite{LP} is  a canonical isomorphism between the two theories $\omega_*$ and  $\Omega_*$.

The algebraic cobordism rings over a point $\Omega_*(\text{Spec }k)$  and  $\omega_*(\text{Spec }k)$ are abbreviated by $\Omega_*(k)$ and $\omega_*(k)$ respectively.
Morel and Levine showed that  $\Omega_*(k)$ is isomorphic to the Lazard ring $\LL_*$, which implies $\LL_*$ is also isomorphic to $\omega_*(k)$. 
It is well known that $\LL_* \otimes_{\ZZ} \QQ$ has a basis formed by products of projective spaces, hence

\begin{align} \label{eq:n}
\omega_*(k) \otimes_{\ZZ} \QQ \cong
\bigoplus_{\lambda=(\lambda_1, \cdots, \lambda_r)} \QQ[\PP^{\lambda_1}\times \cdots\times\PP^{\lambda_{r}}]\end{align} 
where  $\lambda$ runs over all partitions of nonnegative integers.  

\subsection{Bundles on varieties}\label{bundles}
Algebraic cobordism theory was first generalized to pairs $[S,L]$ for a smooth surface $S$  and  a line bundle $L$ in \cite{LP}. 
The resulting algebraic cobordism group $\omega_{2,1}(\CC)$ can be used to prove generating series of Donaldson-Thomas invariants \cite{LP} and it plays a central role in the proof of  \Gott's conjecture in \cite{Tz}. 
More generally, Lee and \Pand\ \cite{LeeP} constructed the algebraic cobordism theory of pairs of smooth varieties and  vector bundles. We first recall their definitions and main results. 

For $X\in \sch$, let $\cM_{n,r}(X)$ be the set of isomorphism classes  $[f:Y \to X, E]$ with $Y\in \sm$  of dimension $n$, $E$ a vector bundle on $Y$ of rank $r$, and $f$ projective. 
$\cM_{n,r}(X)$ is a monoid under disjoint union of domains
and let $\cM_{n,r}(X)^+$ denote the group completion of $\cM_{n,r}(X)$. 

Suppose the projective morphism $\pi: Y\to X\times \PP^1$ together with $\pi_2= p_2\circ \pi: Y\to \PP^1$ give a double point degeneration over $X$ and $E$ is a vector bundle over $Y$. Denote the restriction of $E$ to $Y_i$ by $E_i$ for $i=0,1,2$ and the pullback of the restriction of $E$ to $B$ via the morphism $Y_3\to B$ by $E_3$.  
The \textit{double point relation} of bundles over varieties defined by $\pi$ and $E$ over $X$  is 
$$[Y_0\to X, E_0]-[Y_1\to X, E_1]-[Y_2\to X, E_2]+[Y_3\to X, E_3]\,.$$

Denote $\cR_{n,r}(X)\subset \cM_{n,r}(X)^+$ to be subgroup generated by all double point relations over $X$. The \textit{algebraic cobordism group}   of bundles of rank $r$ over varieties of dimension $n$ is defined by the quotient $$\omega_{n,r}(X):=\cM_{n,r}(X)^+/\cR_{n,r}(X).$$ 

The graded sum $$\omega_{*,r}(X):=\bigoplus_{n=0}^{\infty} \omega_{n,r}(X)$$
is an $\omega_*(k)$-module via product (and pullback). 
It  is also a $\omega_*(X)$-module   if $X\in \sm$.

The main result of \cite{LeeP} is  a natural basis and  dimension of $\omega_{n,r}(k)\otimes_{\ZZ} \QQ$. 
 A \textit{partition pair} of size $n$ and type $r$ is a pair $(\lambda, \mu)$ where
\begin{enumerate}
	\item $\lambda$ is a partition of $n$,
	\item $\mu$ is a subpartition of $\lambda$ of length $l(\mu)\leq r$. 
\end{enumerate}

The second condition means $\mu$ is obtained by deleting some parts of $\lambda$. 
In particular, $\mu$ is allowed to be empty or equal to $\lambda$ if $l(\lambda)\leq r$.
Two subpartitions are equivalent if they only differ by permuting equal parts of $\lambda$. 

Let $\cP_{n,r}$ be the set of equivalent classes of  all partition pairs of size $n$ and type $r$. 
To each $(\lambda, \mu)\in \cP_{n,r}$, we can associate  an element $\phi(\lambda, \mu) \in \omega_{n,r}(k)$  as follows. 
Suppose $\lambda=(\lambda_1, \lambda_2, \ldots, \lambda_{l(\lambda)})$, let 
$\PP^{\lambda}= \PP^{\lambda_1}\times \cdots \times \PP^{\lambda_{l(\lambda)}}$. 
To each part $m$ of $\mu$, let $L_m$ be the line bundle on $\PP^{\lambda}$ obtained by pulling back
$\cO_{\PP^m}(1)$  via the projection $P^{\lambda}\to \PP^m$. Define 

$$\phi(\lambda, \mu)=\left[\PP^{\lambda}, \cO_{\PP^{\lambda}}^{r-l(\mu)}\oplus \left(\bigoplus_{m\in \mu} L_m\right) \right].$$ 

Lee and \Pand\ \cite[Theorem~1]{LeeP} proved that for any nonnegative integer $n$ and $r$,
\begin{align} \label{eq:nr} \omega_{n,r}(k)\otimes_{\ZZ} \QQ =\bigoplus_{(\lambda, \mu)\in \cP_{n,r}}
\QQ\cdot \phi(\lambda, \mu).
\end{align}

As a result, $\{\phi(\lambda, \mu)| (\lambda, \mu)\in \cP_{n,r} \}$ form a basis of $\omega_{n,r}(k)\otimes_{\ZZ} \QQ$ and the dimension of $\omega_{n,r}(k)\otimes_{\ZZ} \QQ$ is the cardinality of $\cP_{n,r}$. 
On the other hand, the dimension of  $\omega_{n,r}(k)\otimes_{\ZZ} \QQ$ is also equal to total number of degree $n$ monomials in Chern classes of dimension $n$ varieties and rank $r$ vector bundles because Chern numbers given by these polynomials are the only invariants of $\omega_{n,r}(k)\otimes_{\ZZ} \QQ$ \cite[Theorem~4]{LeeP}. 
The special case of $r=0$ reduces to Equation \eqref{eq:n}. 

Moreover,  $\omega_{*,r}(k)$ is a free $\omega_*(k)$-module with basis
$$\omega_{*,r}(k)=\bigoplus_{\lambda} \omega_*(k) \cdot \phi(\lambda,\lambda)$$
where the sum is over all partitions $\lambda$ of length at most $r$ \cite[Theorem~2]{LeeP}. 

Once we understand the structure of $\omega_{*,r}(k)$,  $\omega_{*,r}(X)$ is simply an extension of scalars of $\omega_*$ because 
$$\omega_*(X)\otimes_{\omega_{*}(k)} \omega_{*,r}(k) \to \omega_{*,r}(X)$$ is an isomorphism \cite[Theorem~3]{LeeP}. 

\subsection{Lists of bundles}
To prove Equation \eqref{eq:nr}, Lee and \Pand\ defined the algebraic cobordism theory of lists of line bundles $\omega_{n,1^r}$. 
The construction of $\omega_{n,1^r}$ is almost identical to $\omega_{n,r}$ with only two exceptions. First, the elements in $\omega_{n,1^r}(X)$ are isomorphism classes of  the form $[f: Y\to X, L_1, \cdots, L_r]$, where  $L_1, \ldots, L_r$ is an ordered list of line bundles on $Y$. 
Second, the double point relation 
\begin{align} \label{eq:n1r} 
[Y_0\to X, \{L_{0j}\}_{j=1}^r]-[Y_1\to X, \{L_{1j}\}_{j=1}^r]
-[Y_2\to X, \{L_{2j}\}_{j=1}^r]+[Y_3\to X, \{L_{3j}\}_{j=1}^r]
\end{align}
is given  by a double point degeneration from $\pi: Y\to X\times \PP^1$  and an  ordered list of line bundles $L_1,\ldots ,L_r$ on $Y$ such that $L_{ij}$ are line bundles on $Y_i$ from $L_j$ in the same manner of
Section \ref{bundles}.

More generally, one can analogously define algebraic cobordism theory  of lists of vector bundles on varieties to be 
$$\omega_{n, r_1, \ldots, r_k}(X) =  \cM_{n, r_1, \ldots, r_k}(X)^+ /\cR_{n, r_1, \ldots, r_k}(X).$$ 
Here $\cM_{n, r_1, \ldots, r_k}(X)^+$ is the free abelian group over $\ZZ$ generated by isomorphism classes of projective morphisms and vector bundles $[f: Y\to X, E_1,\ldots, E_k]$ with rank $E_i=r_i$ and $\dim Y = n$. 
 The double point relation is the same as \eqref{eq:n1r} except $L_i$ and $L_{ij}$ are replaced by vector bundles $E_i$ and $E_{ij}$.
$\cR_{n, r_1, \ldots, r_k}(X)$ is the subgroup generated by double point relations.
 Therefore $\omega_{n,1^k}(X)$ is a special case of $\omega_{n,\ri}(X)$ with all $r_i=1$. 

The graded sum $$\omega_{*,\ri}(X):=\bigoplus_{n=0}^{\infty} \omega_{n,\ri}(X)$$
is an $\omega_*(k)$-module via product (and pullback). 
If $X\in \sm$, then $\omega_{*,\ri}(X)$ is also a module over the ring $\omega_*(X)$.

\begin{defn}
A \textit{partition list} of size $n$ and type $(\ri)$ is an ordered list $(\lambda,\mu_1, \ldots, \mu_k)$ where
\begin{enumerate}
	\item $\lambda$ is a partition of $n$,
	\item $\mu_i$ is a partition of length $l(\mu_i)\leq r_i$, 
	\item the union of $\mu_1,\ldots, \mu_k$ is a subpartition of $\lambda$. 
\end{enumerate}
\end{defn}

Two partition lists $(\lambda,\mu_1, \ldots, \mu_k)$ and $(\lambda',\mu_1', \ldots, \mu_k')$ are \textit{equivalent} if $\lambda$, $\lambda'$ are the same partitions and permuting equal parts of $\lambda$ makes $\mu_i'$ become  $\mu_i$. 
Let $\cP_{n,r_1,\ldots, r_k}$ be the set of partition lists of size $n$ and type $(\ri)$. For example,
\begin{equation*}\cP_{3,2,1}=\left\{\begin{aligned}
&(3, \varnothing, \varnothing), (3, 3, \varnothing ), (3,\varnothing, 3), \\
&(21, \varnothing,\varnothing), (21, 2, \varnothing), (21, 1, \varnothing), (21, \varnothing, 2),
(21, \varnothing, 1),\\
&(21, 21, \varnothing), (21, 2, 1), (21,1,2)\\
&(111, \varnothing,\varnothing), (111, 1, \varnothing), (111, \varnothing, 1), (111, 11, \varnothing), (111, 1,1),
(111, 11, 1)
\end{aligned}
\right\}\end{equation*}

To each $(\lambda,\mu_1, \ldots, \mu_k) \in \cP_{n,r_1,\ldots r_k}$, we associate an element
$$\phi(\lambda,\mu_1, \ldots, \mu_k) = \left[\PP^{\lambda}, 
\cO_{\PP^{\lambda}}^{r_1-l(\mu_1)}\oplus \left(\bigoplus_{m\in \mu_1} L_{m}\right), \ldots,
\cO_{\PP^{\lambda}}^{r_k-l(\mu_k)}\oplus \left(\bigoplus_{m\in \mu_k} L_{m}
\right) \right]$$
in $\omega_{n,\ri}(k)$ (recall $L_m$ be the line bundle on $\PP^{\lambda}$ obtained by pulling back
$\cO_{\PP^m}(1)$  via the projection $P^{\lambda}\to \PP^m$). Since the union of $ \mu_i$'s is a subpartition of $\lambda$, all line bundles $L_m$ in $\phi(\lambda,\mu_1, \ldots, \mu_k)$ come from distinct factors of $\PP^{\lambda}$. 

The following theorem describes a basis of  $\omega_{n,\ri}(k)\otimes_{\ZZ} \QQ$. 
\begin{theo}\label{nriQ}
For all nonnegative integers $n, \ri$, we have
$$\omega_{n,\ri}(k)\otimes_{\ZZ} \QQ
=\bigoplus_{(\lambda,\mu_1, \ldots, \mu_k) \in \cP_{n,\ri}} \QQ\cdot\phi(\lambda,\mu_1, \ldots, \mu_k). $$
\end{theo}

\begin{proof}
	Let $\fm\subset \omega_*(k)$ be the ideal generated by varieties of positive dimension. 
	Define the graded quotient 
	$$\widetilde{\omega}_{*,\ri}(X)=\omega_{*,\ri}(X)/\fm\cdot \omega_{*,\ri}(X),\,
	\widetilde{\omega}_{*,\ri}(X)=\bigoplus_{n=0}^{\infty} \widetilde{\omega}_{n,\ri}(X)$$
	from the  $\omega_*(k)$-module structure of $\omega_{*,\ri}(X)$. 
	Lemma 13 in \cite{LeeP} and the discussion after it imply that for any smooth variety $X$ and vector bundle $E$ of rank $r$ on $X$, there exists a birational morphism  $\pi: \hat{X}\to X$ such that 
	$$[X \to X,E]=\pi_*[\hat{X}\to \hat{X},L_1\oplus\cdots\oplus L_r]$$ 
	in $\widetilde{\omega}_{*,r}(X)$ for some line bundles $L_1, \ldots, L_r$ on $\hat{X}$. 
	 
	For $r$ bundles, this trick can be applied $r$ times so that every time the pullback of a new $E_i$ splits, every $[X\to X, E_1,\ldots, E_k]$  can be expressed as
	$$[X\to X, E_1,\ldots, E_k] = [\bar{X}\to X, \oplus_{j=1}^{r_1} L_{1j}, \ldots, \oplus_{j=1}^{r_k} L_{kj}] $$ in $\widetilde{\omega}_{*,\ri}(X)$ for some $\bar{X}$ and line bundles $L_{ij}$ on $\bar{X}$. 
	After pushing forward to $\Spec k$,  
	\begin{align} \label{eq:split}[X, E_1,\ldots, E_k] = [\bar{X}, \oplus_{j=1}^{r_1} L_{1j}, \ldots, \oplus_{j=1}^{r_k} L_{kj}] \end{align} 
	in $\widetilde{\omega}_{*,\ri}(k)$. 
		
	Mapping  $[f: Y \to X, L_{11}, \ldots, L_{kr_k}]$ to $[f:Y\to X, \oplus_{j=1}^{r_1} L_{1j}, \ldots, \oplus_{j=1}^{r_k} L_{kj}] $ defines a group homomorphism from $\cM_{n, 1^{\sum r_i}}(X)^+ $  to $\cM_{n, r_1, \ldots, r_k}(X)^+ $.
	Double point relations for  $[Y, L_{11}, \ldots, L_{kr_k}]$  in $\cM_{n, 1^{\sum r_i}}(X)^+ $ naturally impose double point relations for  $[Y, \oplus_{j=1}^{r_1} L_{1j}, \ldots, \oplus_{j=1}^{r_k} L_{kj}] $  by taking direct sums of line bundles. 
	Therefore there is a group homomorphism from  $\omega_{*, 1^{\sum r_i}}(k)$ to $\omega_{*,\ri}(k)$ which is compatible with the $\omega_*(k)$-modules structures. 
	Moreover, Equation \eqref{eq:split} implies the induced homomorphism $\widetilde{\omega}_{*, 1^{\sum r_i}}(k)$ to $\widetilde{\omega}_{*,\ri}(k)$ is surjective.
	It follows from \cite[Proposition~10]{LeeP} that $\widetilde{\omega}_{n,\ri}(k)$ is generated by 
	$$\label{*1}\{\phi(\lambda,\mu_1, \ldots, \mu_k)\, |\, (\lambda, \mu_1, \ldots, \mu_k)\in \cP_{n,\ri}, \cup_i\mu_i=\lambda\}$$ over $\ZZ$. 
	
	We now prove Theorem \ref{nriQ} by induction on $n$. The $n=0$ case is trivial. Assume the result is true for all $n'<n$, then $\oplus_{n'<n}(\omega_{n-n'}(k)\cdot \omega_{n',\ri}(k)) \otimes_{\ZZ}\QQ$ is spanned by 
	$$\label{*2}\{\phi(\lambda, \mu_1, \ldots, \mu_k)\,|\, (\lambda, \mu_1, \ldots, \mu_k)\in \cP_{n,\ri}, \cup \mu_i\subsetneq \lambda\}$$ over $\QQ$. Since  $\omega_{n,\ri}(k)\otimes_{\ZZ} \QQ$ is spanned by generators of  $\omega_{n-n'}(k)\cdot \omega_{n',\ri}(k)$ for $n'<n$  and $\widetilde{\omega}_{n,\ri}(k)$, it is generated by 
	 $$\{\phi(\lambda,\mu_1,\ldots, \mu_k)\,|\,(\lambda,\mu_1,\ldots, \mu_k) \in \cP_{n,\ri} \}. $$
	As a result, the dimension of $\omega_{n,\ri}(k)\otimes_{\ZZ} \QQ$ is less or equal to the size of $\cP_{n,\ri}$.
	We will see   $\dim \omega_{n,\ri}(k)\otimes_{\ZZ} \QQ \geq |\cP_{n,\ri}|$ in the next Theorem so we conclude these generators form a basis of $\omega_{n,\ri}(k)\otimes_{\ZZ} \QQ$.
\end{proof}

	Let $\cC_{n,\ri}$ be the finite-dimensional $\QQ$-vector space of graded degree $n$ polynomials in Chern classes 
	$$\{c_i(T_Y)\}_{i=1}^{n}, \{c_i(E_1)\}_{i=1}^{r_1},\ldots, \{c_i(E_k)\}_{i=1}^{r_k}$$
	(deg $c_i = i$). 
	
\begin{theo}\label{chern} For all $\Theta\in \cC_{n,\ri}$,  the Chern invariants of the list $[Y, E_1, \ldots ,E_k]$
$$\int_Y \Theta (c_1(T_Y),\ldots, c_n(T_Y), c_1(E_1),\ldots, c_{r_1}(E_1),  c_1(E_2),\ldots,c_{r_k}(E_k))$$
respect algebraic cobordism. The resulting map 
$$\omega_{*,\ri}(k)\otimes_\ZZ \QQ\to \cC_{n,\ri}^*$$ is an isomorphism. 
\end{theo}

\begin{proof}
First we prove the Chern invariants respect algebraic cobordism. The argument is due to Pandharipande \cite{Pandprivate}. 

If $E$ is a vector bundle on $Y$ of rank $r$, let $p:\PP(E)\to Y$ be the projective bundle of $E$ over $Y$.
The integral on $Y$ can be lifted to $\PP(E)$ because
\begin{align}\begin{split}\label{projbundle}&\int_Y \Theta (c_1(T_Y),\ldots, c_n(T_Y), c_1(E),\ldots, c_{r}(E)) \\ = &\int_{\PP(E)} \Theta (c_1(p^*T_Y),\ldots, c_n(p^*T_Y), c_1(p^*E),\ldots, c_{r}(p^*E)) c_1(\cO_{\PP(E)}(1)^{r-1}).  \end{split}\end{align}
Since there are exact sequences
$$0 \to \cO_{\PP(E)}\to \cO_{\PP(E)}(1)\otimes p^*E \to T_{\PP(E)/Y}\to 0$$ 
$$0 \to T_{\PP(E)/Y}\to T_{\PP(E)}\to p^*T_Y\to 0,$$
Chern classes of $p^*T_Y$ can be expressed as Chern classes of $\cO_{\PP(E)}(1)$, $ p^*E$, and $T_{\PP(E)}$. 

The first part of \cite[Proposition 5]{LeeP} proves the case for arbitrary number of line bundles.
For vector bundle $E_1, \ldots, E_k$ on $Y$, let $\FF_0 = Y$ and $\pi_0$ be the identity map of $Y$. 
We will construct a sequence of varieties $\FF_i$ with morphisms $\pi_i: \FF_i \to Y$. 
For all $i \geq 1$, let $\FF_i$ be the complete flag variety over $\FF_{i-1}$   obtained from the vector bundle $\pi_{i-1}^*(E_i)$. 
The map $\pi_i$ is the composition of the structure map $\FF_i\to \FF_{i-1}$ and $\pi_{i-1}$. 
Let $\pi = \pi_k$ and $\FF = \FF_k$, since  $\pi: \FF\to Y$ is compositions of structure maps of projective bundles and for all $i$ there are filtrations of vector bundles
$$0 = F_{i,0}\subset F_{i,1}\subset F_{i,2}\subset\cdots \subset F_{i,r_i}=\pi^*E_i $$ with $F_{i,j}/F_{i,j-1} \cong L_{i,j}$ are line bundles.
In $K$-theory of $\FF$ the tangent bundle $T_{\FF}$ is the sum of $\pi^*T_Y$ and some line bundles $B_i$ and vector bundles $\pi^*E_i$ is the sum of  $L_{i,j}$.  
As a result, the Chern invariants of $T_Y$, $E_1$,\ldots, $E_k$ on $Y$can be expressed as integrals of Chern classes of $T_{\FF}$, $B_i$ and $L_{i,j}$ on $\FF$. 

If $Y \to \PP^1$ with vector bundles $\cE_{j}$ gives the double point relation
$$[Y_0, \{E_{0j}\}_{j=1}^k]-[Y_1, \{E_{1j}\}_{j=1}^k]-[Y_2, \{E_{2j}\}_{j=1}^k]+[Y_3, \{E_{3j}\}_{j=1}^k]. $$
The complete flag variety $\FF\to \PP^1$ with $\cE_{j}$ constructed above gives a double point relation of the complete flag variety $\FF_{Y_i}$ constructed above for $Y_i$ with $\{E_{ij}\}$. 
Therefore the Chern invariants of $T_{Y_i}$ and $E_{ij}$ on $Y_i$, which equal the Chern invariants of  $T_{\FF_i}$ and some line bundles on $\FF_{Y_i}$, respect algebraic cobordism.

Because Chern invariants respect algebraic cobordism, the integration map descends to a well-defined bilinear map
$$\rho: (\omega_{*,\ri}(k)\otimes_\ZZ \QQ) \times \cC_{n,\ri} \to \QQ.$$ 

To prove the map is an isomorphism, we follow the idea in  \cite[Section~1]{LeeP}. 
Let $\cQ_{n,\ri}$ be the set of partition lists $(\nu, \mu_1, \ldots, \mu_k)$
where
\begin{enumerate}
	\item every $\mu_i$ is a partition of size $|\mu_i|$ and $\sum_{i=1}^k |\mu_i|\leq n$,
	\item the largest part of $\mu_i$ is at most $r_i$,
	\item $\nu$ is a partition of $n-\sum |\mu_i|$. 
\end{enumerate}

Taking the number of parts equal to $m$ in $\mu_i$ to the power of $c_m(E_i)$ and the number of parts equal to $m$ in $\nu$ to the power of $c_m(T_Y)$ defines a bijective correspondence between $\cQ_{n,\ri}$ and the monomials in $\cC_{n,\ri}$. 
Let $C(\nu, \mu_1, \ldots, \mu_k) $ be the image of $(\nu, \mu_1, \ldots, \mu_k)$ in $\cC_{n,\ri}$. 
On the other hand, there is a natural bijection 
$\epsilon: \cQ_{n,\ri}\to \cP_{n,\ri}$ defined by 
$\epsilon(\nu, \mu_1, \ldots, \mu_k)=(\nu\cup (\cup \mu_i^t), \mu_1^t, \ldots, \mu_k^t)$ 
where $\mu_i^t$ is the partition obtained by transposing the Young diagram of $\mu_i$. Hence $\mu_i^t$ has length at most $r_i$. In particular $|\cP_{n,\ri}|=\dim \cC_{n,\ri}$. 

Now we define a partial ordering on the monomials in $\cC_{n,\ri}$ (and on $\cQ_{n,\ri}$ via the bijection). 
If  $F$ and $G$ are two  monomials in $\cC_{n,\ri}$
and 
$$c_1(E_k)^{a_1}\cdots c_{r_k}(E_k)^{a_{r_k}}\ \ \text{and} \ \ c_1(E_k)^{b_1}\cdots c_{r_k}(E_k)^{b_{r_k}}$$
are the factors of $c_*(E_k)$ in   $F$ and $G$.
We say $F<G$ if $a_{r_k}<b_{r_k}$ or $a_j=b_j$ for all $j>i$ and $a_i<b_i$. 
This partial ordering agrees with the one defined on $\cC_{n,r}$ in \cite[Section~1.2.2]{LeeP}. 

If $\bM_{n,\ri}$ is the matrix with rows and columns indexed and ordered by $\cQ_{n,\ri}$ with entries
$$\bM_{n,\ri}[(\nu,\mu_1\ldots, \mu_k), (\nu',\mu'_1\ldots, \mu'_k)]
=\rho(\phi(\epsilon(\nu,\mu_1,\ldots, \mu_k)), C(\nu',\mu'_1,\ldots, \mu'_k)).$$
For dimension reason as in the proof of Lemma 7 in \cite{LeeP} (and only focus on the integration of $c_*(E_k)$),
$$\bM_{n,\ri}[(\nu,\mu_1,\ldots, \mu_k), (\nu',\mu'_1,\ldots, \mu'_k)]=0 $$ if 
$(\nu,\mu_1,\ldots, \mu_k)<(\nu',\mu'_1,\ldots, \mu'_k)$.  
As a result, $\bM_{n,\ri}$ is a block lower triangular matrix and the blocks are determined by all $(\nu,\mu_1\ldots, \mu_k)$ with  the same $\mu_k$. 
As in the proof of \cite[Proposition~8]{LeeP}, the block corresponding to $\mu_k$ is the matrix
$\bM_{n-|\mu_k|,r_1,\ldots, r_{k-1}}$. Since the base cases $k=1,0$ are already established  there, by induction $\bM_{n,\ri}$ is a nonsingular matrix and $\rho$ is an isomorphism. 
\end{proof}

The structure of $\omega_{n,\ri}(k)$ over $\ZZ$ is determined by: 
\begin{theo} \label{nriZ}
For all $r_i\geq 0$, $\omega_{*,\ri}(k)$ is a free $\omega_*(k)$-module with basis
$$\omega_{*,\ri}(k)=\bigoplus_{(\mu_1, \ldots, \mu_k)} \omega_*(k)\cdot \phi(\cup \mu_i,\mu_1, \ldots, \mu_k)$$
where the sum is over all partitions $(\mu_1, \ldots, \mu_k)$ with length $l(\mu_i)\leq r_i$. 
\end{theo}

\begin{proof}

	In the proof of Theorem \ref{nriQ} we have seen $\widetilde{\omega}_{n,\ri}(k)$ can be generated by 
		$$\{\phi(\lambda,\mu_1, \ldots, \mu_k)\, |\, (\lambda, \mu_1, \ldots, \mu_k)\in \cP_{n,\ri}, \cup_i\mu_i=\lambda\}$$ over $\ZZ$. 
		Hence $\omega_{n,\ri}(k)$ is generated by this set and the subgroups 
		$$\omega_{n}(k)\cdot \omega_{0,\ri}(k),\ldots, \omega_{1}(k)\cdot \omega_{n-1,\ri}(k).$$
	Since $\omega_*(k)\cong \LL$ and it is well known that $\LL_i$ is a free $\ZZ$-module of rank equal to the number of partitions of $i$. 
	By induction the generators of $\LL_i$ multiplied by  $\phi(\cup \mu_i,\mu_1, \ldots, \mu_k)\in \omega_{n-i, r_1,\ldots,r_k}$ gives $ |\cP_{n,\ri}|$ elements to generate  $\omega_{n,\ri}(k)$ over $\ZZ$ (the $n=0$ case is trivial).
	From Theorem \ref{nriQ}, $\omega_{n,\ri}(k)\otimes_{\ZZ}\QQ$ has dimension $|\cP_{n,\ri}|$ thus there are no relations among these generators.  
	\end{proof}

In fact, $\omega_{n,\ri}(k)$ and $\omega_*(X)$ determines $\omega_{n,\ri}(X)$  for general $X\in \sch$. Let $p_Y$, $p_{\PP^{\lambda}}$ be the projections of $Y\times \PP^{\lambda}$ to $Y$ and $\PP^{\lambda}$ respectively. 
\begin{theo}\label{thm:scalar}
The natural map 
$$\gamma_X: \omega_*(X)\otimes_{\omega_{*}(k)} \omega_{*,\ri}(k) \to \omega_{*,\ri}(X)$$ defined by 
\begin{align*}&\gamma_X([Y\xrightarrow{f} X]\otimes \phi(\cup \mu_i,\mu_1, \ldots, \mu_k)) \\
=& \left[Y\times \PP^{\cup \mu_i} \xrightarrow{f\circ p_Y} X, 
\cO^{r_1-l(\mu_1)}\oplus p_{\PP^{\cup \mu_i}}^*\left(\bigoplus_{m\in \mu_1}L_{m}\right), 
\ldots,\cO^{r_k-l(\mu_k)}\oplus p_{\PP^{\cup \mu_i}}^*\left(\bigoplus_{m\in \mu_k} L_{m}
\right) \right] \end{align*} is an isomorphism of $\omega_*(k)$-modules. 
Therefore the algebraic cobordism theory of bundles on varieties $\omega_{*,\ri}$ is an extension of scalars of the original algebraic cobordism theory $\omega_*(X)$,
\end{theo}
\begin{proof} Consider an arbitrary element $[Y\to X, E_1, \ldots, E_k]$ in $\omega_{*,\ri}(X)$. 
We have proved	$$[Y\to Y, E_1,\ldots E_k] = [\bar{Y}\to Y, \oplus_{j=1}^{r_1} L_{1j}, \ldots, \oplus_{j=1}^{r_k} L_{kj}] $$ in $\widetilde{\omega}_{*,\ri}(Y)$ for some $\bar{Y}$ and line bundles $L_{ij}$ on $\bar{Y}$.
Taking suitable direct sum of line bundles in the surjectivity argument and applying the pairing $C_{\psi}$ for every bundle in the injectivity argument of 
\cite[Section~4]{LeeP}  prove the result. 
\end{proof}

\section{Algebraic cobordism theory of bundles and divisors on varieties}\label{sect:algcobordism}
The main goal of the paper is the enumeration of singular subvarieties from sections of bundles with given tangency conditions with a smooth divisor. 
In this section we will study how the data of the problem-- varieties, bundles, divisors-- can be degenerated.  
In other words, we will discuss the algebraic cobordism theory of bundles and  divisors on varieties.


\subsection{Constructions} 
\begin{assu}\label{divassu}
Let $D_1$, \ldots, $D_s$ be a set of reduced effective Cartier divisors  on a smooth variety.
Each $D_i$ is the union of disjoint smooth divisors or the zero divisor. 
The intersection of any $k$ elements from the set of irreducible components of  $D_1$, \ldots, $D_s$ is either empty or smooth of codimension $k$. 
\end{assu}

For $X\in \sch$, let $\cN_{n,1^s,\ri}(X)$ be the set of  the isomorphism classes of lists $[f:M\to X, D_1,\ldots, D_s, E_1,\ldots, E_k]$, where $f: M \to X$ is a projective morphism with $M\in \sm$ of dimension $n$, $E_i$ are vector bundles of rank $r_i$ on $M$, and $D_1,\ldots, D_s$ is  a set of divisors in $M$ satisfying Assumption \ref{divassu}.
The set $\cN_{n,1^s,\ri}(X)$ is graded by the dimension of $M$ and let $\cN_{n,1^s,\ri}(X)^+$ be the group completion of $\cN_{n,1^s,\ri}(X)$ over $\ZZ$. 
For simplicity we usually write the list $[f:M\to X, D_1,\ldots, D_s, E_1,\ldots, E_k]$ by $[M\to X, D_i, E_j]$  and $[f:M\to \Spec k, D_1,\ldots, D_s, E_1,\ldots, E_k]$ by $[M, D_i, E_j]$.

Suppose a  projective morphism $\pi: Y\to X\times \PP^1$ together with $\pi_2= p_2\circ \pi: Y\to \PP^1$ give a double point relation 
$$ [Y_0\to X]-[Y_1\to X]-[Y_2\to X]+[Y_3\to X]$$
over $X$. 
Recall $Y_0:=\pi_2^{-1}(0)$ is the smooth fiber  over $0$,  $\pi_2^{-1}(\infty) = Y_1\cup Y_2$, $B=Y_1\cap Y_2$ and $Y_3$ is a $\PP^1$-bundle over $B$.

Let $\{\cE_i\}_{i=1}^k$  be vector bundles on $Y$ and  $\{\cD_j\}_{j=1}^s$ be  effective divisors on $Y$.
Assume  all $\cD_j$ intersect $Y_i$  and $B$  transversally along smooth divisors.
Denote the restriction of $\cE_j$ to $Y_i$ by $E_{ij}$ for $i=0,1,2$ and call $E_{3j}$ the pullback of the restriction of $\cE_j$ to $B$ via the morphism $Y_3\to B$.
Denote the intersection $\cD_j\cap Y_i$ by $D_{ij}$ for $i=0,1,2$ and call $D_{3j}$ the inverse image of $\cD_j\cap B$ via the morphism $Y_3\to B$. 

Suppose  $\{D_{ij}\}_{j=1}^s$ satisfies Assumption \ref{divassu} for all $i$. 
The \textit{double point relation} given by $(\pi, \{\cD_j\}_{j=1}^s, \{\cE_j\}_{j=1}^k)$ is 
\begin{align} \begin{split} \label{eq:d1ri} 
[Y_0\to X, \{D_{0j}\}, \{E_{0j}\}]-[Y_1\to X, \{D_{1j}\}, \{E_{1j}\}]\\
-[Y_2\to X, \{D_{2j}\},\{E_{2j}\}]+[Y_3\to X, \{D_{3j}\}, \{E_{3j}\}].
\end{split} \end{align}

Denote the  subgroup of $\cN_{n,1^s,\ri}(X)^+$ generated by  double point relations \eqref{eq:d1ri} by $\cQ_{n,1^s,\ri}(X)$. 
The \textit{algebraic cobordism group of divisors and bundles  over varieties} is defined by $$\nu_{n,1^s,\ri}(X)=\cN^+_{n,1^s,\ri}(X)/\cQ_{n,1^s,\ri}(X).$$

The graded sum $$\nu_{*,1^s,\ri}(X):=\bigoplus_{n=0}^{\infty} \nu_{n,1^s,\ri}(X)$$
is an $\omega_*(k)$-module via product (and pullback). 

Let $$\cP'_{n, s, \ri} = \left\{(\lambda; \pi_1, \ldots, \pi_s; \mu_1,\ldots, \mu_k) \,\left|(\lambda, \pi_1, \ldots, \pi_s, \mu_1,\ldots, \mu_k)\in \cP_{n, 1^s, \ri}  \right.\right\}.$$
Consider an element $(\lambda; \pi_1, \ldots, \pi_s; \mu_1,\ldots, \mu_k)\in \cP_{n, s, \ri} $\footnote{We will abbreviate $(\lambda; \pi_1, \ldots, \pi_s; \mu_1, \ldots, \mu_k)$ by  $(\lambda; \pi_i; \mu_j)$ in the following discussion.}. 
Since $\pi_i$ is a partition of length at most $1$, it is either a positive integer $m$ or the empty set $\varnothing$. 
If $\pi_i=\varnothing$,  set $H_i$ to be the divisor $0$. 
If $\pi_i=(m)$, let $H_i$ be the inverse image of a hyperplane of $\PP^m$ via the projection map $\PP^{\lambda}\to \PP^m$. 
Define $$\psi: \cP'_{n, s, \ri} \to \nu_{n,1^s,\ri}(k)$$ by $$\psi(\lambda; \pi_i; \mu_j) =\left[\PP^{\lambda}, H_1,\ldots, H_s, \cO_{\PP^{\lambda}}^{r_1-l(\mu_1)}\oplus \left(\bigoplus_{m\in \mu_1} L_{m}\right), \ldots,
\cO_{\PP^{\lambda}}^{r_k-l(\mu_k)}\oplus \left(\bigoplus_{m\in \mu_k} L_{m}
\right) \right].$$

\begin{theo} \label{isom}

The natural map $$[f: M\to X, D_i, E_j]\to [f: M\to X, \cO(D_i), E_j]$$ induces an isomorphism
$$\nu_{n,1^s,\ri}(X)\to \omega_{n,1^s,\ri}(X)$$
as $\omega_*(k)$-modules.
\end{theo}

 It is easy to see the map  sends $\psi(\lambda; \pi_i; \mu_j)$ to $\phi(\lambda, \pi_i, \mu_j)$.
Therefore Theorem \ref{isom} together with Theorem \ref{nriQ} and Theorem \ref{nriZ} give a natural basis of  $\nu_{n,1^s,\ri}(k)$. 
\begin{coro}\label{vdim} The dimension of $\nu_{n,1^s,\ri}(k)\otimes_{\ZZ} \QQ$ is the number of monomials of degree $n$ in the Chern classes of tangent bundle $T_M$, $\cO(D_i)$ and $E_j$. Furthermore, for all $s, r_i\geq 0$,
$$\nu_{n,1^s,\ri}(k)\otimes_{\ZZ} \QQ
=\bigoplus_{(\lambda; \pi_1, \ldots, \pi_s; \mu_1,\ldots, \mu_k) \in \cP'_{n, s, \ri} } \QQ\cdot\psi(\lambda; \pi_1, \ldots, \pi_s; \mu_1,\ldots, \mu_k). $$ 
The  $\omega_*(k)$-module $\nu_{*,1^s,\ri}(k)$ is  free with basis
$$\nu_{*,1^s,\ri}(k)=\bigoplus_{(\pi_1, \ldots, \pi_s, \mu_1, \ldots, \mu_k)} \omega_*(k)\cdot \psi(\cup \pi_i\cup \mu_i;\pi_1, \ldots, \pi_s;\mu_1, \ldots, \mu_k)$$
where the sum is over partitions $(\pi_1, \ldots, \pi_s, \mu_1, \ldots, \mu_k)$ with length $l(\pi_i)\leq 1$, $l(\mu_i)\leq r_i$. 
\end{coro}

\subsection{Proof of Theorem \ref{isom}}

\begin{lemm}
The map  $[f:M\to X, D_i, E_j]$ to $[f: M\to X, \cO(D_i), E_j]$ sends  double point relations to double point relations. Hence 
 $$\nu_{n,1^s,\ri}(X)\to \omega_{n,1^s,\ri}(X)$$ is a well-defined morphism of abelian groups and $\omega_*(k)$-modules. 
\end{lemm}
\begin{proof} 
If $(\pi, \cD_i, \cE_j)$ gives a double point relation \eqref{eq:d1ri}, then $(\pi, \cO(\cD_i), \cE_j)$ gives a double point relation in $\omega_{n,1^s,\ri}(X)$. 
This is because $\cD_j$ intersect transversally with $Y_i$ and  $B$ along smooth divisors  and  $\cO(\cD_j)|_{Y_i}=\cO(D_{ij})$, $\cO(\cD_i)|_{B}=\cO(\cD_i \cap B)$. 
The addition and $\omega_*(k)$-module structure on both sides are compatible so the map induces a  morphism of abelian groups and $\omega_*(k)$-modules from $\nu_{n,1^s,\ri}(X)\to \omega_{n,1^s,\ri}(X)$.  
\end{proof}

\begin{lemm}
The group morphism  
 $\nu_{n,1^s,\ri}(k)\to \omega_{n,1^s,\ri}(k)$ is surjective. \end{lemm}
\begin{proof}
From  Theorem \ref{nriZ}, $\omega_{n,1^s,\ri}(k)$ is generated by $\omega_*(k)\cdot\phi(\lambda, \pi_i, \mu_j)$ over $\ZZ$.
Since  $\phi(\lambda, \pi_i, \mu_j)$ is the image of  $\psi(\lambda; \pi_i; \mu_j)$ the map is surjective.  
\end{proof}

\begin{prop} \label{n1riZ}
For $s, r_i\geq 0$, $\nu_{*,1^s,\ri}(k)$ is generated by  $$\left\{\left.\omega_*(k)\cdot \psi(\cup \pi_i \cup \mu_j; \pi_1, \ldots, \pi_s; \mu_1,\ldots, \mu_k)\right. | \ l(\pi_i)\leq 1, l(\mu_j)\leq r_j\right\} $$ over $\ZZ$. \end{prop}
\begin{proof}

Let  $P$ be the subgroup of $\nu_{n,1^s,\ri}(k)$ generated by elements of the form $[\PP_V(\cO\oplus N), \{p^{-1}D_i\}_{i=1}^{s-1},  s, p^*{E_j} ]$ where $V$ is a smooth variety of dimension $n-1$, $N$ is a line bundle   and $E_j $ are vector bundles of rank $r_j$ on $V$,  $D_i$ are divisors on $V$, and  $s$ the section at infinity of $p: \PP_V(\cO\oplus N)\to V$.

Consider an element $[M, \{D_i\}_{i=1}^s, \{E_j\}_{j=1}^k]$ in  $\nu_{n,1^s,\ri}(k)$. 
If $D_s\neq 0$ and $D$ is an irreducible component of $D_s$, the deformation to the normal cone of $D$ yields a double point relation 
\begin{align*}&[M, \{D_i\}_{i=1}^s, \{E_j\}_{j=1}^k]
-[M, \{D_i\}_{i=1}^{s-1}, D_s-D,  \{E_j\}_{j=1}^k] \\
-&[\PP_{D},  \{p^{-1}(D_i\cap D)\}_{i=1}^{s-1}, D_{\infty}, \{p^*{E_j|_{D}}\}_{j=1}^k  ]
+[\PP_{D}, \{p^{-1}(D_i\cap D)\}_{i=1}^{s-1}, 0, \{p^*{E_j|_{D}}\}_{j=1}^k ]\end{align*}
where $p: \PP_{D}:=\PP(\cO\otimes N_{M/{D}})\to D$  and $D_{\infty} \cong D$ is the section of $\PP_{D}$ at infinity. 
By induction on the number of components of $D_s$,  $\nu_{n,1^s,\ri}(k)$ is generated by elements with $D_s=0$ and  $P$. 

Now we treat elements in $P$. 
Suppose 
$$[V_0, D_{0i}, N_0,  E_{0j}]-[V_1, D_{1i}, N_1,  E_{1j}]
-[V_2, D_{2i}, N_2,E_{2j}]+[V_3, D_{3i}, N_3, E_{3j}]$$ is a double point relation in $\cN^+_{n-1, 1^{s-1}, 1, r_1, \ldots, r_k}(k)$  given by a double point degeneration $\pi: \cV\to \PP^1$, divisors $\cD_i$ of $\cV$, a line bundle $\cN$ and vector bundles $\cE_j$ on $\cV$. 
Then
$\PP_{\cV}(\cO\oplus N)\to \PP^1$, the section $s=\PP_{\cV}(\cO_{\cV})$ and the pullback of $\cD_i$ and $\cE_j$ yield
a double point relation 
 \begin{align*}[\PP_{V_0}(\cO\oplus N_0),p^{-1}D_{0i} ,s_0, p^*(E_{0j})]
 -[\PP_{V_1}(\cO\oplus N_1), p^{-1}D_{1i}, s_1, p^*(E_{1j})]\\
-[\PP_{V_2}(\cO\oplus N_2), p^{-1}D_{2i}, s_2, p^*(E_{2j})]
+[\PP_{V_3}(\cO\oplus N_3), p^{-1}D_{3i}, s_3, p^*(E_{3j})].
\end{align*}
As a result, $P$ naturally inherits double point relations from   $ \cN^+_{n-1, 1^{s-1},1, \ri}(k)$ so that  the map
$$[V, \{D_i\}_{i=1}^{s-1}, N, E_j]\to 
[\PP_V(\cO\oplus N), \{p^{-1}D_i\}_{i=1}^{s-1},  s, p^*{E_j} ].$$
defines a surjective homomorphism from $\nu_{n-1, 1^{s-1},1, \ri}(k)$ to $P$.

By induction on $s$, $\nu_{n-1, 1^{s-1},1, \ri}(k)$ is generated by $\omega_*(k)\cdot \psi(\cup \pi_i \cup \mu_j; \pi_i; \pi_s,  \mu_j)$ which has the form $[M]\cdot [\PP^{\cup \pi_i \cup \mu_j}, H_i, L_s, E_j]$. 
If $L_s$ is trivial,  $[M]\cdot [\PP^{\cup \pi_i \cup \mu_j}, H_i, L_s, E_j]$ maps to $$[M]\cdot [\PP^{\cup \pi_i \cup \mu_j}\times \PP^1, p^{-1}H_i, \PP^{\cup \pi_i \cup \mu_j}\times \{\infty\}, p^*E_j].$$ 
If $L_s$ is not trivial, i.e. when $\pi_s=m>0$,   $[M]\cdot [\PP^{\cup \pi_i \cup \mu_j}, H_i, L_s, E_j]$ maps to $$[M]\cdot [\PP_{\PP^{\cup \pi_i \cup \mu_j}}(\cO\oplus L_m), p^{-1}H_i, s, p^*E_j].$$
where $p: \PP_{\PP^{\cup \pi_i \cup \mu_j}}(\cO\oplus L_m) \to \PP^{\cup \pi_i \cup \mu_j}$. 
In fact $\PP_{\PP^{\cup \pi_i \cup \mu_j}}(\cO \oplus L_m)$
is the product of $\PP^{\lambda'}$ and  $\PP_{\PP^m}(\cO \oplus \cO_{{\PP^m}}(1)) $ where $\lambda'$ is the union of all $\pi_i$ and $\mu_j$ except $\pi_s$. 
The degeneration to the normal cone of a hyperplane $H_{m+1}$ in $\PP^{m+1}$ yield the double point relation 
$$[\PP^{m+1}, H_{m+1}]-[\PP^{m+1}, 0]-[\PP_{\PP^m}(\cO \oplus \cO_{{\PP^m}}(1)), s]+[\PP_{\PP^m}(\cO \oplus \cO_{{\PP^m}}(1)), 0]$$
Multiplying the family by $M$ and $\PP^{\lambda'}$ and adding $p^{-1}H_i$, $p^*E_j$ give a double point relation which expresses 
$[M]\cdot [\PP_{\PP^{\cup \pi_i \cup \mu_j}}(\cO\oplus L_m), p^{-1}H_i, s, p^*E_j]$ as the sum of elements with $D_s=0$ or 
in $\left\{\left.\omega_*(k)\cdot \psi(\cup \pi_i \cup \mu_j; \pi_i; \mu_j)\right| 
(\cup \pi_i \cup \mu_j; \pi_i; \mu_j) \in \cP'_{*, s, \ri}\right\}$. Therefore $P$ is spanned by these elements over $\ZZ$. 

The map of trivially inserting a zero divisor  $$[M, \{D_i\}_{i=1}^{s-1}, \{E_j\}_{j=1}^k]
\to [M, \{D_i\}_{i=1}^{s-1}, 0,  \{E_j\}_{j=1}^k] $$
defines an isomorphism from 
$\nu_{n,1^{s-1},\ri}(k)$ to the subgroup of $\nu_{n,1^{s},\ri}(k)$ with $D_s=0$. 
By  induction on $s$, $\nu_{n,1^{s-1},\ri}(k)$  is generated by 
$\omega_*(k)\cdot \psi(\cup \pi_i \cup \mu_j; \{\pi_i\}_{i=1}^{s-1}; \mu_j)$ which are mapped to $\omega_*(k)\cdot \psi(\cup \pi_i \cup \mu_j; \{\pi_i\}_{i=1}^{s-1}, 0; \mu_j)$.
They generate the subgroup of $\nu_{n,1^{s-1},\ri}(k)$ with $D_s=0$ and 
this completes our proof.  
\end{proof}

\begin{prop}
The group morphism  
 $\nu_{n,1^s,\ri}(k)\to \omega_{n,1^s,\ri}(k)$ is injective. \end{prop}

\begin{proof}
$\nu_{*,1^s,\ri}(k)$ is generated by  $$\left\{\left.\omega_*(k)\cdot \psi(\cup \pi_i \cup \mu_j; \pi_i; \mu_j)\right. | \ l(\pi_i)\leq1, l(\mu_j)\leq r_j\right\} $$ over $\ZZ$. 
The morphism sends $\left\{\omega_*(k)\cdot \psi(\cup \pi_i \cup \mu_j; \pi_i; \mu_j)\right\}$ to $\left\{\omega_*(k)\cdot \phi(\cup \pi_i \cup \mu_j, \pi_i, \mu_j)\right\}$ and the latter set has no nontrivial relation in $\omega_{n,1^s,\ri}(k)$ by Theorem \ref{nriZ} . 
\end{proof}
\begin{coro}The morphism $$\nu_{n,1^s,\ri}(k)\to \omega_{n,1^s,\ri}(k)$$ is an isomorphism of abelian groups and $\omega_*(k)$ modules. \end{coro}
\begin{proof}  We have proved the morphism is injective and surjective between abelian groups. 
It is easy to check the morphism is a homomorphism of $\omega_*(k)$-modules. 
\end{proof}
 
\begin{theo}
The natural map 
$$\gamma_X: \omega_*(X)\otimes_{\omega_{*}(k)} \nu_{*,1^s, \ri}(k) \to \nu_{*,1^s, \ri}(X)$$ 
defined by 
\begin{align*}&\gamma_X([Y\xrightarrow{f} X]\otimes
\psi(\cup \pi_i \cup \mu_j; \pi_i; \mu_j) \\
=& \left[Y\times \PP^{\lambda} \xrightarrow{f\circ p_Y} X, 
 p_{\PP^{\lambda}}^{-1}H_i, 
\cO^{r_j-l(\mu_j)}\oplus p_{\PP^{\lambda}}^*\left(\bigoplus_{m\in \mu_j}L_{m}\right)\right]\end{align*}
where $\lambda = \cup \pi_i \cup \mu_j$ is an isomorphism of $\omega_*(k)$-modules and $(\lambda; \pi_i; \mu_j)\in \cP_{*,s,r_1,\ldots, r_k}'$. 
Therefore the algebraic cobordism theory of bundles and divisors on varieties $\nu_{*,1^s, \ri}$ is an extension of scalars of the algebraic cobordism theory $\omega_*$.
\end{theo}
\begin{proof}
By definition the morphism $$\omega_*(X)\otimes_{\omega_{*}(k)} \nu_{*,1^s, \ri}(k) \xrightarrow{\sim} \omega_*(X)\otimes_{\omega_{*}(k)} \omega_{*,1^s, \ri}(k) \xrightarrow{\sim} \omega_{*,1^s, \ri}(X)$$ equals the composition of
$$\omega_*(X)\otimes_{\omega_{*}(k)} \nu_{*,1^s, \ri}(k)  \overset{\gamma_X}{\rightarrow}  \nu_{*,1^s, \ri}(X)  \to  \omega_{*,1^s, \ri}(X)$$
thus $\gamma_X$ is injective. Since $\gamma_X$ is canonically a homomorphism of  $w_*(k)$-modules, the only remaining part is to prove $\gamma_X$ is surjective. 

If $n=0$ then all bundles are trivial and divisors are empty, so $\nu_{0,1^s, \ri}(X) $ equals $\omega_{0, 1^s, \ri}(X)$.
If $s = 0$, there are no divisors so $\nu_{n,1^0, \ri}(X) $ equals $\omega_{n, \ri}(X)$.
In this two cases, the theorem follows from Theorem \ref{thm:scalar}. 

Now we mimic the trick in the first part of Proposition \ref{n1riZ}.
Let  $P$ be the subgroup of $\nu_{n,1^s,\ri}(X)$ generated by elements of the form $[\PP_V(\cO\oplus N)\to X, \{p^{-1}D_i\}_{i=1}^{s-1},  s, p^*{E_j} ]$ where $V$ is a smooth variety of dimension $n-1$, $\PP_V(\cO\oplus N)\to X$ is the composition of the structure map $p: \PP_V(\cO\oplus N)\to V$ and a map $f: V\to X$, $N$ is a line bundle   and $E_j $ are vector bundles of rank $r_j$ on $V$,  $D_i$ are divisors on $V$, and  $s$ is the section at infinity.

Consider an element $[f: M\to X, \{D_i\}_{i=1}^s, \{E_j\}_{j=1}^k]$ in  $\nu_{n,1^s,\ri}(k)$. 
If $D_s\neq 0$ and $D$ is an irreducible component of $D_s$, the map $f$ naturally extends to the total family  of the deformation to the normal cone  $Bl_{D\times \{\infty\}} M\times \PP^1 \to  M\times \PP^1\overset{p_M}{\rightarrow} M \overset{f}{\rightarrow} X$. 
Thus the deformation to the normal cone  yields a double point relation 
\begin{align*}&[f:M\to X, \{D_i\}_{i=1}^s, \{E_j\}_{j=1}^k]
-[\PP_{D} \to X,  \{p^*(D_i\cap D)\}_{i=1}^{s-1}, D_{\infty}, \{p^*{E_j|_{D}}\}_{j=1}^k  ]\\
-&[f:M\to X, \{D_i\}_{i=1}^{s-1}, D_s-D,  \{E_j\}_{j=1}^k] 
+[\PP_{D}\to X, \{p^*(D_i\cap D)\}_{i=1}^{s-1}, 0, \{p^*{E_j|_{D}}\}_{j=1}^k ].\end{align*}
By induction on the number of irreducible components of $D_s$,  $\nu_{n,1^s,\ri}(k)$ is generated by elements with $D_s=0$ and  $P$. 

As in proof of Proposition \ref{n1riZ}, 
$$[f:V\to X, \{D_i\}_{i=1}^{s-1}, N, E_j]\to 
[f\circ p: \PP_V(\cO\oplus N)\to X, \{p^{-1}D_i\}_{i=1}^{s-1},  s, p^*{E_j} ].$$
defines a surjective homomorphism from $\nu_{n-1, 1^{s-1},1, \ri}(X)$ to $\nu_{n, 1^s, \ri}(X)$ which is onto $P$.
By induction on $n$ and $s$, the top morphism in the following commutative diagram is an isomorphism and therefore the image of $\gamma_X$ generates $P$. 
$$\xymatrix{ 
\omega_*(X)\otimes_{\omega_{*}(k)} \nu_{*,1^{s-1},1, \ri}(k) \ar[d]\ar[r]& 
\ar[d] \nu_{*, 1^{s-1},1, \ri}(X) \ar[d]\\
\omega_*(X)\otimes_{\omega_{*}(k)} \nu_{*,1^s, \ri}(k) \ar[r]^(.6){\gamma_X}&  \nu_{*,1^s, \ri}(X) }
$$
(The diagram commutes because $\PP_{Y\times M}(\cO\oplus p_M^*N)\cong Y\times \PP_M(\cO\oplus N)$ if $N$ is  a bundle on $M$.)

Lastly, elements with $D_s=0$ are canonically isomorphic to $\nu_{n, 1^{s-1}, \ri}(X)$. By induction they are also generated by images of $\gamma_X$. 
\end{proof}

\begin{proof}[proof of Theorem \ref{isom}]
$$\xymatrix{ 
\omega_*(X)\otimes_{\omega_{*}(k)} \nu_{*,1^s, \ri}(k) \ar[d]\ar[r]& 
\ar[d] \nu_{n, 1^s, \ri}(X) \ar[d]\\
\omega_*(X)\otimes_{\omega_{*}(k)} \omega_{*,1^s, \ri}(k) \ar[r]&  \omega_{*,1^s, \ri}(X) }
$$
We have proven the top, bottom and left morphisms are isomorphisms therefore the right is also an isomorphism. 
\end{proof}
\section{Notations for singularity and tangency conditions}\label{sect:notations}
\subsection{Singularities}  First, we recall some results in singularity theory from \cite{Loo}. 
Let  $(X,x)$ be  an isolated complete intersection singularity (ICIS)  of dimension $p$. Suppose $(X,x)$ can be locally embedded in $\CC^N$ so that  $X$ is  locally defined by  common zeros of $f = (f_1, \ldots, f_{N-p}): (\CC^N,0)\to (\CC^{N-p}, 0)$.
The  miniversal deformation space of the ICIS  $(X,x)$ exists and its dimension is equal to the  Tjurina number  $$\tau(f):=\dim_{\CC}\cO_{\CC^N,0}^{N-p}/(Df\cdot \cO_{\CC^N,0}^N+\langle f_1, f_2, \ldots, f_{N-p}\rangle \cO_{\CC^N,0}^{N-p}).$$  
Two map-germs $f, g: (\CC^N,0)\to (\CC^{N-p}, 0)$ define the same ICIS if  they are $\mathscr{K}$-equivalent, i.e.  there exists a commutative diagram
\begin{displaymath}
    \xymatrix{ (\CC^N,0) \ar[r]^f \ar[d]  & (\CC^{N-p}, 0) \ar[d] \\
               (\CC^N,0) \ar[r]^g 		& (\CC^{N-p}, 0) }
\end{displaymath}
in which the vertical maps are biholomorphisms.

It is well known \cite{Wall} that every ICIS is \textit{finitely determined}, which means that if $f$ defines an ICIS then there exists a finite integer $k$ such that any map-germ  $g$ with the same $k$-jet defines the same ICIS.
If this holds then we say the  ICIS is \mbox{$k$-determined}.

\begin{defn} Let $\delta=(X,x)$ be an ICIS and $\fm_{X,x}$ be the maximal ideal in the local ring $\cO_{X,x}$. Let 
\begin{align*} k(\delta) &= \min_k \{k\,|\,\delta \text{ is } k\text{-determined}\} \text{ (the degree of determinacy)}, \\
 \xi(\delta)&=  \cO_{X,x}/\fm_{X,x}^{k(\delta)+1},
 \end{align*} 
 and $N(\delta)$ be the length of the zero-dimensional scheme $\xi(\delta)$. It is easy to see the isomorphism classes of $\xi(\delta)$, $k(\delta)$ and $N(\delta)$ can be computed from any representation of $\delta$ so they are indeed invariants of the ICIS.  
\end{defn}


Let  $\delta=(\delta_1, \delta_2,\ldots, \delta_{|\delta|})$ be  a collection of ICIS.  
A variety $X$ has singularity type $\delta$ if $X$ is singular at exactly $|\delta|$  points $\{x_1,x_2,\ldots, x_{|\delta|}\}$ and the singularity at $x_i$ is exactly $\delta_i$. 
Such variety  must contain a subscheme isomorphic to $\coprod_{i=1}^{|\delta|}\xi(\delta_i)$.  
In general ICIS includes smooth point, but  $\delta$ is assumed to have no smooth point. 

Define 
\begin{align*}
 N(\delta)=\sum_{i=1}^{|\delta|} N(\delta_i), \,\,\text{ and} \,\,
\tau(\delta)=\sum_{i=1}^{|\delta|} \tau(\delta_i). 
\end{align*}
The number $\tau(\delta)$ is the expected number of conditions for a subvariety to have singularity type $\delta$.

\subsection{Tangency conditions}

For Caporaso-Harris invariants, general nodal curves are tangent to the fixed line at smooth points of the curves. 
However, in higher dimension certain tangency conditions must occur at singular points of subvarieties. 
For example, consider an ICIS  defined by $f: (\CC^2, 0) \to (\CC^2, 0)$, $f(x,y) = (x^2, y^3)$. 
Let $C$ be a curve in $\CC^3$ defined by $g_1=g_2=0$.  
If the intersection of $C$ and the $xy$-plane has this ICIS at a point $p$, then the Jacobian matrix of $g_1$ and $g_2$ at $p$ has rank at most one.  
Therefore $C$ must be singular at $p$.  

In general, consider an ICIS defined by $r$ equations. 
If the Jacobian matrix of those $r$ equations has rank less than $r-1$ at a point, then any codimension $r$ subvariety having tangency type with a divisor of this ICIS  must be singular at the point, since the Jacobian matrix of the defining equations has rank less than $r$.  
Then this point can be counted as a tangency point or a singular point,  and there are many possible singularity types.  
One may think these singularities are ``induced'' by the tangency condition. 
The motivation of the  definition of non-induced singularity type in Section \ref{sect:non-induced} is  to avoid this ambiguity. 

Let $Y$ be a smooth variety of dimension $n$, $E$ be a vector bundle of rank $r$ and $D$ be a smooth divisor in $Y$.
If $E$ is globally generated,  the zeros of a generic section of $E$ is $(n-r)$-dimensional. 
Recall in introduction we explained why the tangency conditions are recorded by two collections of ICIS $\alpha$ (at assigned points) and $\beta$ (at unassigned points). 
Here we further assume the dimensions of all ICIS in $\alpha$ and $\beta$ are $n-r-1$ and give the formal definition of tangency conditions below.

Denote the numbers of ICIS in $\alpha$ and $\beta$ by $|\alpha|$ and $|\beta|$ and index them by $$\alpha=(\alpha_1, \alpha_2, \ldots, \alpha_{|\alpha|}), \beta=(\beta_1, \beta_2, \ldots, \beta_{|\beta|}).$$ 
 Fix an ordered list of $|\alpha|$ distinct points 
$$\Omega =\{p_i\}_{1\leq i\leq |\alpha|} \subset D.$$ 
If $s$ is a section of $E$, assume its zero locus $X$ is a  reduced subvariety of dimension $n-r$. 
We say the subvariety $X$ satisfies \textit{tangency conditions} $(\alpha, \beta)$ with $D$ or has \textit{tangency type} $(\alpha, \beta)$ with $D$ at $\Omega$ if there are $|\beta|$ distinct points $\{q_j\}_{1\leq j\leq |\beta|}$ in $D\backslash \Omega$, such that the scheme-theoretical intersection of $D$ and $X$ has singularity type $\alpha_i$ at $p_i$ and $\beta_j$ at $q_j$. 
$\alpha$ may contain smooth point as an ICIS because  this condition just means passing through a fixed point on $D$. 
But $\beta$ will not contain smooth points. 

The scheme-theoretic intersection $D\cap X$ has a given ICIS at some point $p\in D$ if and only if the restriction of $s$  on $D$  defines the ICIS at $p$.
Since $D$ is smooth everywhere, the expected codimension of this condition is the Tjurina number $\tau$ of this ICIS.
If the point $p$ is given, then the expected codimension becomes $\tau+n-1$.
As a result, the expected codimension of tangency conditions $(\alpha, \beta)$ at given $\Omega$ is
$$\codim(\alpha, \beta):= (n-1)|\alpha|+\tau(\alpha)+\tau(\beta).$$

\begin{exam}\label{multiplicity} If $X$ is a curve, $X$ is tangent to $D$ of multiplicity $k$ at a smooth point if and only if $X\cap D$ has the ICIS $(\Spec \CC[t]/t^k, 0)$.  
The ICIS $(\Spec \CC[t]/t^k, 0)$ is $(k-1)$-determined and  
$$\tau(\Spec \CC[t]/t^k, 0) = k-1, \ \  \xi(\Spec \CC[t]/t^k, 0) = \Spec \CC[t]/t^k, \ \ N(\Spec \CC[t]/t^k, 0) =k.$$ 
The expected codimension for $X$ to be tangent to  $D$ of  multiplicity $k$ at unassigned point is $k-1$. 
The expected codimension for $X$ to be tangent to  $D$ of  multiplicity $k$ at an assigned point is  $n+k-2$. 

\end{exam}

\section{Degeneration formula} \label{sect:deg}
In this section we express the numbers of singular subvarieties  satisfying given tangency conditions as intersection numbers on Hilbert schemes of points  under ampleness assumptions (Proposition \ref{prop:daL}). 
Then two degeneration formulae  (Corollary \ref{goalpha}, \ref{d=dd2}) show the intersection numbers behave nicely under double point relations. 
The degeneration formulae combined with the basis of the algebraic cobordism groups discussed in Section \ref{sect:algcobordism} will lead to the existence of universal polynomials in Section \ref{sec:univ}. 

\subsection{Hilbert schemes}

Recall $Y$ is a smooth variety of dimension $n$, $D$ is a smooth divisor and $E$ is a vector bundle of rank $r$ on $Y$. Let $\alpha$, $\beta$ be  collections of ICIS of dimension $n-r-1$ and  $\delta$ is a collection of ICIS of dimension $n-r$. 

Fix an ordered list of $|\alpha|$ distinct points $\Omega= \{p_i\}_{1\leq i\leq |\alpha|}$ in $D$ and 
let $Y^{[m]}$ be the Hilbert scheme of $m$ points on $Y$. 
Let $$N(\alpha, \beta, \delta) = N(\alpha)+N(\beta)+N(\delta).$$
Define $Y^0_{\Omega}(\alpha,\beta,\delta)$ to be the subset of $Y^{[N(\alpha, \beta, \delta)]}$ consisting of subschemes $\coprod_{k=1}^{|\alpha|+|\beta|+|\delta|} \eta_k$ satisfying the following conditions: 
\begin{enumerate}
	\item If $D$ is the zero divisor, then $Y^0_{\Omega}(\alpha,\beta,\delta)$ is the empty set except when $\alpha$, $\beta$ are both the empty set $\varnothing$. 
	\item If $\alpha$, $\beta$ and  $\delta$ are all empty sets, let $Y^0_{\Omega}(\alpha,\beta,\delta)$  be $\Spec \CC$. 
	\item $\eta_k$ are supported on distinct points of $Y$. 
	\item For $1\leq i\leq |\alpha|$, $\eta_i$ is a subscheme of $D$   support at  $p_{i}$ and   $\eta_i\cong \xi(\alpha_i)$ as schemes.
	\item For  $1\leq i\leq |\beta|$,  $\eta_{|\alpha|+i}$ is a subscheme of $D$ and   $\eta_{|\alpha|+i}\cong \xi(\beta_i)$ as schemes.
	\item For  $1\leq i\leq |\delta|$,  $\eta_{|\alpha|+|\beta|+i}$ is isomorphic to $\xi(\delta_i)$. 

\end{enumerate}

 
For every $k\in \NN$, let 
$Z_k\subset Y\times Y^{[k]}$ be the universal closed subscheme with projections 
$p : Z_k \to Y$, $q: Z_k \to Y^{[k]}$.

\begin{defn}\label{defn:tautological} If $E$ is a vector bundle of rank $r$ on $Y$, define $E^{[k]} = q_* p^*E$. Because $q$ is finite and flat, $E^{[k]}$ is a vector bundle of rank $kr$ on $Y^{[k]}$ and it is called the \textit{tautological bundle} of $E$. 
\end{defn}

Recall $E$ is \textit{$k$-very ample} if for every zero-dimensional subscheme $Z$ of length $k+1$ in $Y$, the natural restriction map $H^0(E)\to H^0(E\otimes \cO_Z)$ is surjective.

For simplicity, let $N=N(\alpha, \beta, \delta)$, $c=\text{codim}(\alpha,\beta)+\tau(\delta)$ be the expected codimension for having non-induced singularity $\delta$ and tangency type $(\alpha,\beta)$ with $D$ at $\Omega$. Let 
$$\Lambda^{\Omega}_{\alpha,\beta, \delta}(Y,D,E)
:=c_{rN-c}(E^{[N]})\cap [Y_{\Omega}(\alpha,\beta,\delta)]$$ 
in the Chow group of $\Hilb{Y}{N}$. 

\begin{lemm} $\Lambda^{\Omega}_{\alpha,\beta, \delta}(Y,D,E)$ is a zero cycle.
\end{lemm}

\begin{proof} Since $\text{codim}(\alpha,\beta)=(n-1)|\alpha|+\tau(\alpha)+\tau(\beta)$,
\begin{align*} & rN-c=r(N(\alpha)+N(\beta)+N(\delta))-(n-1)|\alpha|-\tau(\alpha)-\tau(\beta)-\tau(\delta)\\
=& (rN(\alpha) -(n-1)|\alpha|-\tau(\alpha))
+(rN(\beta)-\tau(\beta))+(rN(\delta)-\tau(\delta))\end{align*}

Next we compute the dimension of $Y_{\Omega}(\alpha,\beta,\delta)$. 
If $f$ is a representative of a single ICIS $\delta$, since $\delta$ is $k(\delta)$-determined, $$\cO_{\CC^n,0}/(\langle f_1, f_2,\ldots, f_r\rangle +\fm_{\CC^n,0}^{k(\delta)+1}) \cong \cO_{\CC^n,0}/(\langle g_1, g_2,\ldots, g_r\rangle +\fm_{\CC^n,0}^{k(\delta)+1})$$  if and only if $g=(g_1,\ldots, g_r)$ also defines $\delta$. 
These two subschemes are equal for general choices of $g_i$ in $\langle f_1, f_2,\ldots, f_r\rangle +\fm_{\CC^n,0}^{k(\delta)+1}/\fm_{\CC^n,0}^{k(\delta)+1}$ and the dimension for all choices of $g_1, \ldots ,g_r$ is 
$$r \cdot \left( \dim_{\CC} (\langle f_1, f_2,\ldots, f_r\rangle +\fm_{\CC^n,0}^{k(\delta)+1})/\fm_{\CC^n,0}^{k(\delta)+1}\right).$$

Let $H$ be the direct sum of $r$ copies of the vector space of degree  $k(\delta)$ polynomials in $x_1, \ldots, x_n$. A deformation of $f$ over $H$ is given by $F(x,h)=f(x)+h(x)$. 
Since $\delta$ is $k(\delta)$-determined, this deformation of $f$ is versal  and the dimension of $H$ is the sum of  $\tau(\delta)$ and the dimension of $k(\delta)$-jets $g$ which also define $\delta$. 
As a result, the dimension of choosing an subscheme $\eta \cong \xi(\delta)$  is 
\begin{align*}&\dim H-\tau(\delta)-r \cdot \left( \dim (\langle f_1, f_2,\ldots, f_r\rangle +\fm_{\CC^n,0}^{k(\delta)+1})/\fm_{\CC^n,0}^{k(\delta)+1}\right) \\
=&r \dim \cO_{\CC^n,0}/(\langle f_1, f_2,\ldots, f_r\rangle +\fm_{\CC^n,0}^{k(\delta)+1})  -\tau(\delta) =rN(\delta)-\tau(\delta).\end{align*} 
Therefore the dimension of choosing $  \coprod_{k=1}^{|\delta|} \eta_{|\alpha|+|\beta|+k}$
is $rN(\delta)-\tau(\delta)$. 

For $\beta$, the only difference is that the subschemes are supported on $D$.
If we change $n$ to $n-1$ above, the same argument holds and the dimension of choosing $  \coprod_{k=1}^{|\beta|} \eta_{|\alpha|+k}$
is $rN(\beta)-\tau(\beta)$. 
For $\alpha$,  the dimension of the subscheme $\coprod_{k=1}^{|\alpha|} \eta_{k}$
is $rN(\alpha)-\tau(\alpha) - (n-1)|\alpha|$ because $\eta_1, \ldots, \eta_{|\alpha|}$ are supported at assigned points on $D$.
So we can conclude the dimension of $Y_{\Omega}(\alpha,\beta,\delta)$ is 
$$ (rN(\alpha) -(n-1)|\alpha|-\tau(\alpha)) + (rN(\beta)-\tau(\beta))+(rN(\delta)-\tau(\delta))=rN-c.$$
Since the computation is local, the dimension is the same everywhere on $Y^0_{\Omega}(\alpha,\beta,\delta)$ and thus on $Y_{\Omega}(\alpha,\beta,\delta)$. 
\end{proof}



Define
$d_{\alpha,\beta, \delta}^{\Omega}(Y, D,E)$ to be the degree of the zero cycle $\Lambda^{\Omega}_{\alpha,\beta, \delta}(Y,D,E)$. 
If $\alpha$ is empty, $\Omega$ is also empty so there is no ambiguity to write $d_{\varnothing,\beta,\delta}(Y,D,E)$.  



\begin{prop}\label{prop:daL} If  $\alpha$,  $\beta$ are two collections of ICIS of dimension $n-r-1$,  and  $\delta$ is a collection of ICIS of dimension $n-r$.
For all  $(N(\alpha)+N(\beta)+N(\delta)+n-r+1)$-very ample vector bundle $E$ of rank $r$  on a smooth variety $Y$ of dimension $n$, if  $\PP(V)\subset \PP(H^0(E))$ is  a general linear subspace of dimension $\text{codim}(\alpha,\beta)+\tau(\delta)$, then there are precisely $d_{\alpha,\beta, \delta}^{\Omega}(Y, D,E)$ subvarieties defined by zero sets of sections in $V$ which have non-induced singularity type $\delta$ and satisfy tangency conditions $(\alpha,\beta)$ with $D$ at $\Omega$. 
\end{prop}
\begin{proof}

Recall  $N=N(\alpha)+N(\beta)+N(\delta)$, $c=\text{codim}(\alpha,\beta)+\tau(\delta)$. 

Let $\{s_0,s_1,\ldots, s_c\}\subset H^0(E)$ be a coordinate system of $V$,  then 
$\{q_*p^* s_i\}_{i=0}^c$
are global sections of $E^{[N]}$ via the evaluation map $H^0(Y,E)\otimes \cO_{Y^{[N]}} \to E^{[N]}$.  
By the ampleness assumption, $H^0(E)\to H^0(E\otimes \cO_Z)$ is surjective for every closed subscheme $Z$ in $\Hilb{Y}{N}$, 
therefore $E^{[N]}$ are generated by sections from $H^0(E)$. 
By applying \cite[Example~14.3.2]{Fu} and a dimensional count,  the locus $W$  where $\{q_*p^* s_i\}_{i=0}^c$ are linearly dependent  is Poincar\'e dual to 
 $c_{rN-c}(E^{[N]})$ for general $V$
and  $$[W\cap Y_{\Omega}(\alpha,\beta,\delta)]=c_{rN-c}(E^{[N]})\cap [Y_{\Omega}(\alpha,\beta,\delta)]=\Lambda^{\Omega}_{\alpha,\beta, \delta}(Y,D,E).$$

Applying \cite[Example~14.3.2]{Fu} again to $Y_{\Omega}(\alpha,\beta,\delta)\backslash Y_{\Omega}^0(\alpha,\beta,\delta)$ with a dimensional count implies that $W\cap Y_{\Omega}(\alpha,\beta,\delta)$   is only supported on $Y^0_{\Omega} (\alpha,\beta,\delta)$ for general $V$. 
If $Z\in Y^{[N]}$, the kernel of $H^0(E)\to H^0(E|_{Z})$ is a $rN$-codimensional subspace of $H^0(E)$ so it defines a morphism $\phi_N: Y^{[N]}\to $ Grass$(h^0(E)-rN, H^0(E))$.    
$W\cap Y_{\Omega}(\alpha,\beta,\delta)$ is the preimage of the Schubert cycle 
$\{U \in \text{ Grass}(h^0(E)-rN, H^0(E))\,|\, U\cap V\neq 0\}$
 under $\phi_N$. 
By \cite{Kltran},  $W\cap Y_{\Omega}(\alpha,\beta,\delta)$  is smooth for general $V$.

For general $V$, every point in $W\cap Y_{\Omega}(\alpha,\beta,\delta)$ corresponds to a section $s\in V$ which  vanishes at a point in $Y^0_{\Omega}(\alpha,\beta,\delta)$.  
We need to  show the vanishing locus of such $s$ must be a subvariety  with non-induced singularity types $\delta$ satisfying tangent conditions  $(\alpha,\beta)$ with $D$ at $\Omega$. 
Denote the zero locus of $s$ by $X$ and let $X$ contain the zero-dimensional subscheme $\eta=\coprod_{i=1}^{|\alpha|+|\beta|+|\delta|} \eta_i$ in   $Y^0_{\Omega}(\alpha,\beta, \delta)$.
The argument is divided into three  steps.

\textbf{Step 1}. For general $V$, we show all such subvarieties $X\in \Lambda^{\Omega}_{\alpha,\beta, \delta}(Y,D,E)$ must have dimension $n-r$ and the  non-induced singular points of  $X$ must lie in  the support of $\coprod_{i=1}^{|\delta|} \eta_{|\alpha|+|\beta|+i}$. In particular it implies that $X$ is reduced. 
We will prove it by contradiction. 

Suppose $X$ has a  singular point disjoint from the support of $\eta$ or the dimension of $X$ is greater than $n-r$. 
then $X$ must contain a zero-dimensional subscheme in $Y'(\alpha,\beta, \delta):= \{ \eta' \cup Z \}$ for all
$$
\eta' \in Y^0_{\Omega}(\alpha,\beta, \delta) \text{ and } Z\cong
\Spec \CC[x_1, \ldots, x_{n-r+1}]/(x_1, \ldots, x_{n-r+1})^2. $$ $Y'(\alpha,\beta, \delta)$ is a subset of $\Hilb{Y}{N+n-r+2}$.
Since $V$ is general and $E$ is $(N+n-r+1)$-very ample,  $$c_{r(N+n-r+2)-c}(E^{[N+n-r+2]})\cap [\overline{Y'(\alpha,\beta, \delta)}]\geq 1$$
because the existence of $X$.  
But the cycle vanishes because 
\begin{align*}&\text{dim }\overline{Y'(\alpha,\beta, \delta)}=\text{dim } Y(\alpha,\beta, \delta)+n+(r-1)(n-r+1) \\= &rN-c+ n+(r-1)(n-r+1)<r(N+n-r+2)-c.\end{align*}
Therefore such $X$ does not exist.

\textbf{Step 2}. We prove that subvarieties in $\Lambda_{\alpha,\beta, \delta}^{\Omega}(Y,D,E)$ must have non-induced singularity type  $\delta$. 
Let $X$, $s$ and $\eta$ be same as above. By Step 1 we know the non-induced  singular points of $X$ lie in   the support of $\coprod_{i=1}^{|\delta|} \eta_{|\alpha|+|\beta|+i}$. 
Because $X$ contains $\coprod_{i=1}^{|\delta|} \eta_{|\alpha|+|\beta|+i}$, the non-induced singularity type of $X$ at supp $\eta_{|\alpha|+|\beta|+i}$ is either $\delta_i$ or ``worse than'' $\delta_i$. 
If the non-induced singularity type of $X$ is worse than $\delta_i$, without loss of generality we can assume the singularity type at the point $y_1$ (which is defined to be the support of $\eta_{|\alpha|+|\beta|+1}$)  is not $\delta_1$.  

Fix an isomorphism between $(Y,y_1)$ with $(\CC^n,0)$. Let $g_1=\cdots=g_{r}=0$ be the local equations of $X\subset Y$ at  $y_1$   and $\eta_{|\alpha|+|\beta|+1}$  be $\Spec\cO_{\CC^n,0} / \langle f_1,\ldots, f_{r}, \fm_{\CC^n,0}^{k(\delta_1)+1}\rangle$. 
Because  $\eta_{|\alpha|+|\beta|+1}$ is a subscheme of $X$, the ideal $\langle f_1,\ldots, f_r, \fm_{\CC^n,0}^{k(\delta_1)+1}\rangle \subset \cO_{\CC^n,0}$ contains but can not equal $\langle g_1,\ldots, g_{r}, \fm_{\CC^n,0}^{k(\delta_1)+1}\rangle$  otherwise by the finite determinacy theorem the singularity type of $X$ at $y_1$ is $\delta_1$. 
Then there must be an ideal $J$ satisfying 
$$\langle g_1,\ldots, g_{r}, \fm_{\CC^n,0}^{k(\delta_1)+1}\rangle\subset J\subset \langle f_1,\ldots, f_{r}, \fm_{\CC^n,0}^{k(\delta_1)+1}\rangle $$
and $\text{dim}_{\CC}(\cO_{\CC^n,0}/J)=\text{length}(\eta_{|\alpha|+|\beta|+1})+1$. 
If we fix $\eta_{|\alpha|+|\beta|+1}$, every ideal $J$ satisfying 
$\fm_{\CC^n,0}^{k(\delta_1)+1}\subset J\subset \langle f_1,\ldots, f_{r},  \fm_{\CC^n,0}^{k(\delta_1)+1}\rangle $ and  
$\text{length}(\cO_{\CC^n,0}/J)=\text{length}(\eta_{|\alpha|+|\beta|+1})+1$
corresponds to  a codimension $1$ subspace in the vector space spanned by the images of $ f_1,\ldots, f_{r}$ in $\cO_{\CC^n,0}/\fm_{\CC^n,0}^{k(\delta_1)+1} $. Therefore the dimension of all choices of $J$ is no greater than $r-1$.

If we allow $\eta_{|\alpha|+|\beta|+1}$ to vary, the dimension of such $J$ is no greater than $r-1$ plus the dimension of all $ \eta_{|\alpha|+|\beta|+1}$.   
Since $X$ contains $\Spec \cO_{Y,y_1}/J$ and $\eta\backslash \eta_{|\alpha|+|\beta|+1}  $,  a similar argument to step 1 implies such $X$ must contribute to intersection of 
$c_{r(N+1)-c}(E^{[N+1]})$ and the closure of all possible union of $\Spec \cO_{Y,y_1}/J$ and $ \coprod_{i\in \{1,2,\ldots,|\alpha|+|\beta|+|\delta|\}\backslash\{|\alpha|+|\beta|+1\}}  \eta_i  $ in $Y^{[N+1]}$. For dimension reason this cycle is empty, 
therefore for general $V$, every $X$ in $\Lambda_{\alpha,\beta, \delta}^{\Omega}(Y,D,E)$ must have  singularity type $\delta$ at the support of $\coprod_{i=1}^{|\delta|}  \eta_{|\alpha|+|\beta|+i}$.

All $\eta$  with  $ \eta_{|\alpha|+|\beta|+1}$ supported on $D$ form a closed subvariety of $Y_{\Omega}(\alpha, \beta, \delta)$. For dimension reason the intersection of this subvariety and $c_{rN-c}(E^{[N]})$ is empty. 
Therefore by symmetry for general $V$, $X$ has non-induced singularity type precisely $\delta$. 

\textbf{Step 3}. We prove for general $V$, subvarieties  in $\Lambda^{\Omega}_{\alpha,\beta, \delta}(Y,D,E)$  must have tangent type $(\alpha,\beta)$ with $D$ at $\Omega$. 

Since $X$ contains  $\coprod_{i=1}^{|\alpha|+|\beta|} \eta_k \subset D$, the definition of  $\eta_k$ ensures that $X$ has tangency type ``at least'' $(\alpha,\beta)$ with $D$ at the support of $\coprod_{i=1}^{|\alpha|+|\beta|} \eta_i$. 
If $X$ is tangent to $D$ at another point, then $D\cap X$ is not  smooth there and the tangent space at this point is at least $(n-r)$-dimensional. Therefore, $X\cap D$ contains a subscheme $Z\cong  \Spec \CC[x_1,\ldots, x_{n-r}]/(x_1,\ldots, x_{n-r})^2$ which is disjoint from $\coprod_{k=1}^{|\alpha|+|\beta|} \eta_k$.  
All such $Z$ form a subset of $D^{[n-r+1]}$ of dimension $n-1+(r-1)(n-r)$ where $n-1$ comes from the choice of point on $D$ and   $(r-1)(n-r)$  from the choice of $x_1,\ldots, x_{n-r}$ in the local coordinate of the point in $D$. 
Using a similar argument in Step 2 and the fact that $E$ is $(N+n-r)$-very ample, those $X$ containing $Z \cup \eta$ in $V$ contribute to intersection of  $c_{r(N+n-r+1)-c}(E^{[N+n-r+1]})$ and  a  $(rN-c+n-1+(r-1)(n-r))$-dimensional subvariety in $Y^{[N+n-r+1]}$. The dimension of this cycle is negative so such $X$ does not exist. 
Therefore we can conclude $X$ is tangent to $D$ at exactly $|\alpha|+|\beta|$ points. 

If the tangency type of $X$ and $D$ is still not $(\alpha, \beta)$, it must due to a different ICIS at one of the tangency points on $D\cap X$. 
Suppose this happens at an unassigned point, which can be assumed to be $q = \text{supp } \eta_{|\alpha|+1}$ without loss of generality.
The argument in Step 2 can be applied here with $Y$ replaced by $D$. 
It shows $X$ must contain the union of  $\Spec \cO_{D,q}/J$ with $ \coprod_{i\in \{1,2,\ldots,|\alpha|+|\beta|+|\delta|\}\backslash\{|\alpha|+1\}}  \eta_i$ for some $J$ such that   the length of $\Spec \cO_{D,q}/J$ is $N(\beta_1)+1$. 
Since $E$ is $N$-very ample,  such $X$ must contribute to intersection of 
$c_{r(N+1)-c}(E^{[N+1]})$ and the closure of all possible  $\Spec \cO_{D,q}/J$ union $ \coprod_{i\in \{1,2,\ldots,|\alpha|+|\beta|+|\delta|\}\backslash\{|\alpha|+1\}}  \eta_i  $ in $Y^{[N+1]}$. 
All possible $\Spec \cO_{D,q}/J$ form a subset of dimension no greater than $r-1$ plus the dimension of all $\eta_{|\alpha|+1}$. 
Thus for dimension reason this cycle is empty and tangency type of  $X$ with $D$ at unassigned points must be $\beta$. 

If the tangency type of $D$ and $X$  is not $\alpha_k$ at an assigned point $p_k$,  we can modify the argument above by considering  $q=p_k$. 
Then all possible $\Spec \cO_{D,p_k}/J$ still form a subset of dimension no greater than $r-1$ plus the dimension of all $\eta_{k}$. 
Thus for dimension reason it can not happen and the tangency type of $X$ with $D$ must be $(\alpha, \beta)$ at $\Omega$.
\end{proof}


\subsection{Relative Hilbert schemes}\label{rel} To set up, we recall some useful facts about the relative Hilbert schemes constructed by Li and Wu \cite{LW}. 
Let $U$ be a smooth irreducible curve and $\infty \in U$ be a specialized point. Consider a flat projective family of schemes $\pi: \cY \to U$  which satisfies
\begin{enumerate}
	\item $\cY$ is smooth and $\pi$ is smooth away from the fiber $\pi^{-1}(\infty)$; 
	\item $\pi^{-1}(\infty)=: Y_1\cup Y_2$  is a union of two irreducible smooth components $Y_1 $ and $Y_2$ which intersect transversally along a smooth divisor $B$.
\end{enumerate}
In \cite{LW}, Li and Wu constructed a family of Hilbert schemes of $n$ points $\pi^{[n]}: \cYn\to U$, whose smooth fiber over $t\neq \infty$ is $Y_t^{[n]}$, the Hilbert scheme of $n$ points  on the smooth fiber $Y_t$ of $\pi$. 
To compactify this moduli space, one can replace $\cY$ by a new space $\cY[n]$ so that $\cY$ and $\cY[n]$ have same smooth fibers over $t\neq \infty$, but over $\infty$  the fiber of $\cY[n]$ is a semistable model 
$$\cY[n]_{\infty}=Y_1 \cup \Delta_1 \cup \Delta_2\cup \ldots \Delta_{n-1}\cup Y_2,$$ 
where $\Delta_i\cong \PP_B(\cO_B\oplus N_{Y_1/B})$. 
The fiber of $\cYn$ over $\infty$ consists of length $n$ zero-dimensional subschemes supported on the smooth locus of $\cY[n]_{\infty}$. 
Since $\cY[n]_{\infty}$ is a chain, any zero-dimensional subscheme $Z$ on $\cY[n]_{\infty}$ can be decomposed into $Z_1\cup Z_2$ where $Z_1$ is supported on $Y_1 \cup \Delta_1 \cup \ldots \cup \Delta_i$ and $Z_2$ is supported on $\Delta_{i+1} \cup \ldots \cup \Delta_{n-1}\cup Y_2$ for some $i$. 
These  $Z_1$ and $Z_2$ belong to the relative Hilbert scheme $(Y_i/B)^{[k]}$ and the decomposition gives $$\cY^{[n]}|_{\infty} = \cup_{k=0}^n (Y_1/B)^{[k]}\times (Y_2/B)^{[n-k]}. $$
Li and Wu proved that the moduli stack $\cYn$ is a separated and proper Deligne-Mumford stack of finite type over $U$. 

Suppose  $U = \PP^1$ and vector bundle  $\cE$  on $\cY$ with $\pi$ defines a double point relation, let $E_i$ denote the restriction of $\cE$ to $Y_i$ for $i=0,1,2$ and call $E_{3}$ the pullback of the restriction of $\cE$ to $B$ via the morphism $Y_3\to B$. 
The tautological bundle $\cEn$ on $\cYn$ exists.
A straightforward generalization of \cite[Lemma~3.7]{Tz} shows the restriction of $\cEn$ on smooth fibers $Y_t^{[n]}$ is the tautological bundle  $(\cE|_{Y_t})^{[n]}$ and the restriction  on  $(Y_1/B)^{[k]} \times (Y_2/B)^{[n-k]}$ is the direct sum of tautological bundles on relative Hilbert schemes\footnote{More precisely, there is a natural way to extend bundles to the expanded relative pairs of $Y_1\cup_B Y_2$. See Section 2.3 of \cite{LW}.} $E_1^{[k]}\oplus E_2^{[n-k]}$. 

Let $\cD$ be a smooth effective divisor on $\cY$ which intersects $Y_i$ and $B$ transversally along smooth divisors. 
Denote the intersection $\cD\cap Y_i$ by $D_i$ for $i=0,1,2$ and call $D_{3}$ the inverse image of $\cD\cap B$ via the morphism $Y_3\to B$.
Because $\cD$ intersects $B$ transversally and $\cY[n]$ is constructed by blowups, $\cD$ can be naturally lifted to a divisor $\cD[n]$ in $\cY[n]$ and its restriction $\cD[n]_{\infty}$ on $\cY[n]_{\infty}$ are union  of smooth divisors on $Y_1$, $Y_2$ and all $\Delta_i$.

Recall in Section 4.3 the subsets $Y^0_{\Omega}(\alpha,\beta, \delta)$ and $Y_{\Omega}(\alpha,\beta, \delta)$ of $Y^{[N(\alpha)+N(\beta)+N(\delta)]}$ are defined.
If the following assumption is satisfied,  these subsets can be extended to the family $\cY\to U$.

\begin{assu}\label{sections} Assume there are   sections $\{\sigma_{i}\}_{1\leq i\leq |\alpha|}$  of the family $\cD\to U$ whose images are smooth curves in $\cD$ and disjoint from  $B$. In addition, assume these sections are disjoint  over an open subset of $U$ which contains $0$ and $\infty$. By replacing $U$ with this open subset,  we can assume the sections are disjoint everywhere.  
\end{assu}
Call  
$\Omega_0 =\{\sigma_{i}(0)\}_{1\leq i\leq |\alpha|}$,  
 $\Omega_1 =\{\sigma_{i}(\infty)\}_{1\leq i\leq |\alpha|}\cap D_1$,   
$\Omega_2 =\{\sigma_{i}(\infty)\}_{1\leq i\leq |\alpha|} \cap D_2$.
Let $\alpha_j$ be collections of ICIS satisfying $\alpha_0=\alpha$ and $\alpha = \alpha_1\cup \alpha_2$. 
On the relative Hilbert scheme $(Y_j/B)^{[N(\alpha_j)+N(\beta_j)+N(\delta_j)]}$, 
the subset  $(Y_j/B)^0_{\Omega_j}(\alpha_j, \beta_j, \delta_j)$ can be defined in a similarly way to  $Y^0_{\Omega}(\alpha,\beta, \delta)$ except now all subschemes are in the relative Hilbert scheme of $Y_j/B$ and $\coprod_{i=1} ^{|\alpha_j|+|\beta_j|} \eta_i  $ are supported on $\cD[n]_{\infty}$ for some $n$. 
Let $(Y_j/B)_{\Omega_j}(\alpha_j, \beta_j, \delta_j)$ be the closure of $(Y_j/B)^0_{\Omega_j}(\alpha_j, \beta_j, \delta_j)$ with reduced induced structure and 
$d_{\alpha_j,\beta_j, \delta_j}^{\Omega_j} (Y_j/B, D_j, E_j)$ be the degree of the zero cycle 
$$c_{r(N(\alpha_j)+N(\beta_j)+N(\delta_j))-\codim(\alpha_j,\beta_j)-\tau(\delta_j)}(E_j^{[N(\alpha_j)+N(\beta_j)+N(\delta_j)]})\cap [(Y_j/B)_{\Omega_j}(\alpha_j, \beta_j, \delta_j)].$$

\begin{lemm}\label{1-cycle} If $\cY\to U$ satisfies Assumption \ref{sections}. For every $\alpha$, $\beta$, and $\delta$, there is a flat $1$-cycle  $\cY(\alpha,\beta, \delta) \subset \cY^{[N(\alpha)+N(\beta)+N(\delta)]}$ over $U$ such that 
\begin{align*} \cY(\alpha,\beta, \delta) \cap (Y_0)^{[N(\alpha)+N(\beta)+N(\delta)]} =(Y_0)_{\Omega_0}(\alpha,\beta, \delta)\text{ and } \\
\cY(\alpha,\beta, \delta)\cap \left({(Y_1/B)^{[m]}\times (Y_2/B)^{[N(\alpha)+N(\beta)+N(\delta)-m]}}\right)
=  \\ \bigcup (Y_1/B)_{\Omega_1}(\alpha_1,\beta_1, \delta_1)\times (Y_2/B)_{\Omega_2}(\alpha_2,\beta_2, \delta_2),  \end{align*}
where the union is over all  $\beta_j$  and  $\delta_j$ satisfying $N(\alpha_1)+N(\beta_1)+N(\delta_1)=m$,  $\beta=\beta_1 \cup \beta_2$ and $\delta=\delta_1\cup\delta_2$.
.  
\end{lemm}

\begin{proof}
For simplicity we write $N=N(\alpha)+N(\beta)+N(\delta)$ and $\Omega_t=\{\sigma_{i}(t)\}$.
Let $\cY^0(\alpha,\beta, \delta)$ be the union of all $(Y_t)^0_{\Omega_t}(\alpha,\beta, \delta)$ for all smooth fibers $Y_t$ over $t\in U$. 
Define $\cY(\alpha,\beta, \delta)$ to be the closure of $\cY^0(\alpha,\beta, \delta)$ in $\Hilb{\cY}{N}$, then by definition $\cY(\alpha,\beta, \delta) \cap Y_0^{[N]} =(Y_0)_{\Omega_0}(\alpha,\beta, \delta)$.
Since the fibers of $\pi$ are smooth away from $B$,  for every point $p$ in $\cY\backslash B$ one can choose an analytic neighborhood $V_p$ such that $\pi: V_p \to U$ is a trivial fibration over its image in $U$. 
Since $\cD$ and the image of sections $\sigma_{i}$ are a submanifolds of $\cY$, we can further assume $\cD$ is  the zero locus of a coordinate function of $V_p$ if $p\in D$ and if $p$ is in the image of a section $\sigma_{i}$ then the image locally gives the coordinate along the direction of $U$.

Assume $z$ is a point in $\cY(\alpha,\beta, \delta)$ which lies over $\infty$. 
Write $z=\cup z_k$  so that each $z_k$ is only supported at one point  $p_k$.
Then every $p_k$ must  belong to the smooth locus of a component $\Delta$ of  a semistable model $\cY[n]_{\infty}=Y_1 \cup \Delta_1 \cup \Delta_2\cup \ldots \Delta_{n-1}\cup Y_2$. 
Since $z$ is in the closure of  $\cY^0(\alpha,\beta, \delta)$, there is a sequence of points $\{z_j\}$ in $\cY^0(\alpha,\beta, \delta)$ approaching $z$. 
By shrinking $V_{p_k}$ if necessary and taking  large enough $j$, we can assume  $\{V_{p_k}\}$ are pairwise disjoint and all $z_j$ are contained in $\cup_k  V_{p_k}$.
As a result, the sequence $\{z_j\}$ can be decomposed into disjoint sequences $z_{jk}\to z_k$ in  $V_{p_k}$. 
Since $V_{p_k}\subset \cY \to U$ is a trivial fibration over its image, there is another trivial projection in the complementary direction called
$q_k: V_{p_k} \to \Delta $ which projects $V_{p_k}$ along the direction of $U$ onto a small open neighborhood of $p_k$ on  $\Delta$.  

Because $z_{jk}$ are subschemes of fibers,  the trivial projection $q_k$  does not change their isomorphism types. 
In addition, our choice of coordinate functions assures that  $q_k$ sends points on $\cD$ to $\cD$ and  $\sigma_{i}(t)$ to $\sigma_{i}(\infty)$.     
Therefore $\{\cup_k q_k(z_{jk})\}$ is a sequence in $(Y_1/B)^0_{\Omega_1}(\alpha_1,\beta_1, \delta_1)\times (Y_2/B)^0_{\Omega_2}(\alpha_2,\beta_2, \delta_2)$ for some  $\beta_j$, $\delta_j$ and $\Omega_j$ and its limit $z$ is in the closure 
$(Y_1/B)_{\Omega_1}(\alpha_1,\beta_1, \delta_1)\times (Y_2/B)_{\Omega_2}(\alpha_2,\beta_2, \delta_2)$.

On the other hand,  $(Y_1/B)^0_{\Omega_1}(\alpha_1,\beta_1, \delta_1)\times (Y_2/B)^0_{\Omega_2}(\alpha_2,\beta_2, \delta_2)$ is a subset  of $\cY(\alpha,\beta, \delta)$  because the trivialized neighborhood allows subschemes  to be moved away from the fiber of $\infty$ along fibers of $q_k$. So for every point in $(Y_1/B)^0_{\Omega_1}(\alpha_1,\beta_1, \delta_1)\times (Y_2/B)^0_{\Omega_2}(\alpha_2,\beta_2, \delta_2)$, we can create a sequence in  $\cY^0(\alpha,\beta, \delta)$ approaching it. 
Therefore  $\cY(\alpha,\beta, \delta)$, defined as the closure of  $\cY^0(\alpha,\beta, \delta)$, contains $(Y_1/B)_{\Omega_1}(\alpha_1,\beta_1, \delta_1)\times (Y_2/B)_{\Omega_2}(\alpha_2,\beta_2, \delta_2)$ because the latter is  the closure of $(Y_1/B)^0_{\Omega_1}(\alpha_1,\beta_1, \delta_1)\times (Y_2/B)^0_{\Omega_2}(\alpha_2,\beta_2, \delta_2)$. 
 
The irreducible components of the open part $\cY^0(\alpha,\beta, \delta)$  dominate the curve $U$. 
Therefore the closure of every irreducible component is flat over $U$ and their union   $\cY(\alpha,\beta, \delta)$ is also flat. 
\end{proof}


If $H_1$, $H_2$ are two different hyperplanes in $\PP^n$ and $p$ is a point in $H_2$ but not in $H_1$, then any  $H_1'$, $H_2'$ and $p'$ satisfying the same conditions are the images of $H_1$, $H_2$ and $p$ by an linear coordinate change of $\PP^n$. 
Therefore if $|\alpha|= 1$, $d_{\alpha, \beta, \delta}^{p}(\PP^n/H_1,H_2,\cO^{\oplus r})$ is independent of the choice of $H_i$ and $p$ and $p$ can be dropped in the notation. 
\begin{coro}\label{goalpha}
If $\Omega$ is an ordered list of $|\alpha|$ distinct points on $D$ and $\alpha_1$, $p_1$ are the first singularity and point in $\alpha$ and $\Omega$ respectively. 
Let $\pi:Bl_{p_1} Y \to Y$ be the blowup of $Y$ at the point $p_1$ with exceptional divisor $E_1$, then 
\begin{align*}d_{\alpha,\beta, \delta}^{\Omega}(Y,D,E) &=  \sum d_{\alpha\backslash\alpha_1,\beta_1,\gamma_1}^{\Omega\backslash p_1}(Bl_{p_1}Y/E_1,Bl_{p_1} D,\pi^*E) d_{\alpha_1,\beta_2,\gamma_2}(\PP^n/H_1,H_2,\cO^{\oplus r}),\\
d_{\alpha\backslash\alpha_1,\beta,\gamma}^{\Omega\backslash p_1}(Y,D,E) &=  \sum d_{\alpha\backslash\alpha_1,\beta_1,\gamma_1}^{\Omega\backslash p_1}(Bl_{p_1}Y/E_1,Bl_{p_1} D,\pi^*E) d_{\varnothing,\beta_2,\gamma_2}(\PP^n/H_1,H_2,\cO^{\oplus r}).
\end{align*}
where $H_1$ and $H_2$ are two distinct hyperplanes in $\PP^n$ and both sums are over all $\beta_j$ and  $\delta_j$ satisfying  $\beta=\beta_1\cup\beta_2$ and $\delta=\delta_1\cup\delta_2$.
\end{coro}
\begin{proof}
Let $\cY\to \PP^1$ be the family of blowing up the point $p_1\times {\infty}$ in the trivial family $Y\times \PP^1\to \PP^1$. 
$\cY$ carries the pullback of $E$ via the map $\cY\to Y\times \PP^1\to Y$  and the proper transform of $D\times \PP^1$. 
The vector bundle and divisor on $\cY$ restrict to the fiber over $0$ is $[Y,D,E]$ and over $\infty$ is the union of two components $[Bl_{p_1}Y/E_1,Bl_{p_1} D,\pi^*E]$ and $[\PP^n/H_1,H_2,\cO^{\oplus r}]$. 

$\Omega\backslash p_1\times \PP^1$ induce sections of $\cY\to\PP^1$ which satisfy Assumption \ref{sections}.  
Applying Lemma \ref{1-cycle} to this family and taking degree at $0$ and $\infty$ 
give the second equation. 

A local computation shows the proper transform of the line $\{p_1\}\times \PP^1$  gives a section of $\cY\to \PP^1$ and this section does not intersect $E_1$. 
The union of this section and the sections above satisfy Assumption \ref{sections}. 
Applying Lemma \ref{1-cycle} to this family and taking degree at $0$ and $\infty$ 
give the first equation. 
\end{proof}

If  $\alpha$ is the empty set $\varnothing$,  Assumption \ref{sections} is trivially satisfied. By taking the degree of the $1$-cycle in Lemma \ref{1-cycle} again  we have

\begin{coro} \label{d=dd2} If ($\pi: \cY\to \PP^1$, $\cD$, $\cE$) gives a double point relation 
$$[Y_0, D_0, E_0] - [Y_1, D_1, E_1]-[Y_2, D_2, E_2]+[Y_3, D_3, E_3].$$
Then 
$$d_{\varnothing,\beta,\gamma}(Y_0,D_0,E_0) = 
\sum d_{\varnothing,\beta_1,\gamma_1}(Y_1/B,D_1,E_1) d_{\varnothing,\beta_2,\gamma_2}(Y_2/B,D_2,E_2)$$
where the sum is over all  $\beta_j$ and  $\delta_j$ satisfying  $\beta=\beta_1\cup\beta_2$ and $\delta=\delta_1\cup\delta_2$.
\end{coro}

\section{Universal polynomials and generating series} \label{sec:univ}

In this section, we will combine the two degeneration formulae (Corollary \ref{goalpha} and \ref{d=dd2}) and  the algebraic cobordism group of bundles and divisors on varieties $\nu_{n,1,r}$ in Section \ref{sect:algcobordism} to prove the main Theorems.

For  a single ICIS $\delta$, let $z_{\delta}$ be a formal variable indexed by $\delta$. 
For  a collection of ICIS $\delta=\{\delta_i\}$, we write $z_{\delta}=\prod_i z_{\delta_i}$.  
It is easy to see  $z_{\delta_1}\cdot z_{\delta_2}$ is equal to $z_{\delta}$ if and only if $\delta$ is the union of $\delta_1$ and $\delta_2$ and the multiplication is commutative. 
For a collections of ICIS $\beta$,  variables and their products $y_{\beta}$ are defined in the same way. 

Let $\QQ[[y_{\beta}, z_{\delta}]]$ be the formal power series ring in $y_{\beta}, z_{\delta}$ and $\QQ[[y_{\beta}, z_{\delta}]]^{\times}$ be the multiplicative group of invertible series.

\begin{defn}
For a rank $r$ vector bundle $E$  and a smooth divisor $D$ on a smooth $n$-dimensional variety $Y$, a fixed collection of dimension $n-r-1$ ICIS $\alpha$ and an ordered list of $|\alpha|$ distinct points $\Omega\subset D$, 
define the generating series
$$T_{\alpha}^{\Omega}(Y,D,E)=\sum_{\beta,\delta} d^{\Omega}_{\alpha,\beta,\delta}(Y,D,E) y_{\beta}z_{\delta}. $$
If $B$ is a smooth divisor of $Y$ which intersects $D$ transversally at smooth divisors, let
$$T_{\alpha}^{\Omega}(Y/B,D,E)=\sum_{\beta,\delta} d_{\alpha,\beta,\delta}^{\Omega}(Y/B,D,E) y_{\beta}z_{\delta}.$$
The sums are over all collections of  dimension $n-r-1$ ICIS $\beta$ and collections of  dimension $n-r$ ICIS $\delta$. 
\end{defn}

\begin{rem} 
By definition if  $\alpha$, $\beta$ and $\delta$ are all empty, $d_{\alpha,\beta,\delta}(Y,D,E)=1$. It makes sense because $1$ is the degree of $\PP(H^0(E))$ in itself. 

\end{rem}

\begin{prop}\label{Pi} For each ICIS $\alpha_i$, there exists a series $ P_{\alpha_i}\in \QQ[[y_{\beta}, z_{\delta}]]$ such that 
$$T_{\alpha}^{\Omega}(Y,D,E) = T_{\varnothing}(Y,D,E)\cdot \prod_{\alpha_i\in \alpha} P_{\alpha_i}.$$
As a result, $T_{\alpha}^{\Omega}(Y,D,E)$ and $d_{\alpha}^{\Omega}(Y,D,E)$ are independent of $\Omega$. 
\end{prop}
\begin{proof}
By Corollary \ref{goalpha}, 
\begin{align*}
T_{\alpha}^{\Omega}(Y,D,E) &=  T_{\alpha\backslash\alpha_1}^{\Omega\backslash p_1}(Bl_{p_1}Y/E_1,Bl_{p_1} D,\pi^*E) T_{\alpha_1}(\PP^n/H_1,H_2,\cO^{\oplus r}) \\
T_{\alpha\backslash\alpha_1}^{\Omega\backslash p_1}(Y,D,E) &=  T_{\alpha\backslash\alpha_1}^{\Omega\backslash p_1}(Bl_{p_1}Y/E_1,Bl_{p_1} D,\pi^*E) T_{\varnothing}(\PP^n/H_1,H_2,\cO^{\oplus r})
\end{align*}
Therefore if one lets 
$$P_{\alpha_1} = T_{\alpha_1}(\PP^n/H_1,H_2,\cO^{\oplus r})/ T_{\varnothing}(\PP^n/H_1,H_2,\cO^{\oplus r})$$
then
$$T_{\alpha}^{\Omega}(Y,D,E) =  T_{\alpha\backslash\alpha_1}^{\Omega\backslash p_1}(Y,D,E)
\cdot P_{\alpha_1}.$$
The desired equation follows from applying this equality many times. 
\end{proof}

By Proposition \ref{Pi}, it is safe to write $T_{\alpha}(Y,D,E)$ and $d_{\alpha}(Y,D,E)$ because they are independent of $\Omega$.

\begin{prop} \label{deg formula} Suppose ($\pi: \cY\to \PP^1$, $\cD$, $\cE$) gives a double point relation
$$[Y_0, D_0, E_0] - [Y_1, D_1, E_1]-[Y_2, D_2, E_2]+[Y_3, D_3, E_3]$$
 in the algebraic cobordism group of divisors and bundles constructed in Section 3. Then 
\begin{align}\label{eq:homo}T_{\varnothing}(Y_0, D_0, E_0) = \frac{T_{\varnothing}(Y_1, D_1, E_1) T_{\varnothing}(Y_2, D_2, E_2) }{T_{\varnothing}(Y_3, D_3, E_3) }.\end{align}
In other words,  $T_{\varnothing}$ induces a  homomorphism from the algebraic cobordism group $\nu_{n,1,r}$ to $\QQ[[y_{\beta}, z_{\delta}]]^{\times}$.
\end{prop}

\begin{proof}
Recall $B=Y_1\cap Y_2$. Lemma \ref{d=dd2} implies
$$T_{\varnothing}(Y_0, D_0, E_0) = T_{\varnothing}(Y_1/B, D_1, E_1) T_{\varnothing}(Y_2/B, D_2, E_2).$$

To derive a relation of generating series without relative terms, we  apply this equality to four families: $\cY$, 
the blowup of  $Y_1\times \PP^1$ along $B\times\{\infty\}$, 
the blowup of  $Y_2\times \PP^1$ along $B\times\{\infty\}$,  
and the blowup of  $Y_3\times \PP^1$ along $B\times\{\infty\}$.
The results are
\begin{align*}
T_{\varnothing}(Y_0, D_0, E_0) &= T_{\varnothing}(Y_1/B, D_1, E_1) T_{\varnothing}(Y_2/B, D_2, E_2),\\
T_{\varnothing}(Y_1, D_1, E_1) &= T_{\varnothing}(Y_1/B, D_1, E_1) T_{\varnothing}(Y_3/B_{\infty}, D_3, E_3),\\
T_{\varnothing}(Y_2, D_2, E_2) &= T_{\varnothing}(Y_2/B, D_2, E_2) T_{\varnothing}(Y_3/B_0, D_3, E_3),\\
T_{\varnothing}(Y_3, D_3, E_3) &= T_{\varnothing}(Y_3/B_{\infty}, D_3, E_3) T_{\varnothing}(Y_3/B_0, D_3, E_3).
\end{align*}

Here $B_{\infty}$ and $B_0$ are the section at $\infty$ and $0$ of the $\PP^1$-bundle $Y_3\to B$. 
Multiplying the first and fourth equations and  dividing by the second and third equations prove the proposition.
\end{proof}

Proposition \ref{deg formula} implies that  torsion elements in  $\nu_{n,1,r}(k)$  has generate series $1$ and $T_{\varnothing}(Y,D,E)$ can be determined by a basis of $\nu_{n,1,r}\otimes \QQ$, which corresponds to a basis of Chern numbers of $Y$, $D$, and $E$. 

\begin{theo}\label{thm:Ai} 
If $\{\Theta_1, \Theta_2,\ldots,  \Theta_m \}$ forms a basis of the finite-dimensional $\QQ$-vector space of graded degree $n$ polynomials in the Chern classes of 
		$$\{c_i(T_Y)\}_{i=0}^{n}, c_1(\cO(D)),\{c_i(E)\}_{i=0}^{r}.$$
	Then there exist power series $A_1$, $A_2, \ldots, A_m$ in  $\QQ[[y_{\beta}, z_{\delta}]]^{\times}$ such that 
\begin{align}\label{eq:Ai}T_{\alpha}(Y,D,E)
= \prod_{k=1}^m A_k^{\int_Y \Theta_k(c_i(T_Y), c_1(D), c_i(E))}\cdot \prod_{\alpha_i\in \alpha} P_{\alpha_i}. \end{align}
It follows that $T_{\varnothing}$ (and therefore its coefficients) is an invariant of the algebraic cobordism group $\nu_{n, 1, r}$.
 \end{theo}
\begin{proof}
Proposition \ref{deg formula} shows $T_{\varnothing}$  defines a multiplicative homomorphism from $\nu_{n, 1, r}$ to $\QQ[[y_{\beta}, z_{\delta}]]^{\times}$. Recall $\nu_{n, 1, r}\otimes \QQ$ is isomorphic to the finite-dimensional $\QQ$-vector space of graded degree $n$ polynomials in the Chern classes of 
$\{c_i(T_Y)\}_{i=0}^{n}, c_1(\cO(D)),\{c_i(E)\}_{i=0}^{r}$ (Corollary \ref{vdim}) by sending $[Y,D,E]$ to its Chern numbers. Therefore $T_{\varnothing}$ factors through this vector space and 
$$T_{\varnothing}(Y,D,E)
= \prod_{k=1}^m A_k^{\int_Y\Theta_k(c_i(T_Y), c_1(D), c_i(E))}.$$
The rest of the proof follows from Proposition \ref{Pi}. 
\end{proof}

Even for nodal curves on surfaces, the number of coefficients in universal polynomials grow very fast with the number of conditions which makes it difficult to compute them from special cases.
The following expression, inspired by \cite{KP},  reduces the computation of universal polynomials to finding linear polynomials. 

\begin{coro}\label{exp} 
Let $exp(R)$ be the formal power series 
$\sum_{n=0}^{\infty} \frac{R^n}{n!}$.
The generating series $T_{\varnothing}(Y,D,E)$ has an exponential description
 $$T_{\varnothing}(Y,D,E)=\exp{\left(\sum_{\beta, \delta} \frac{a_{\beta, \delta} y_{\beta}z_{\delta}}{\#Aut(\beta)\cdot\#Aut(\delta)}\right)}$$ so that every  $a_{\beta, \delta}$ is a homogeneous linear polynomial in Chern numbers of $Y$, $D$ and $E$ and $a_{\varnothing, \varnothing}=0$.
\end{coro}
\begin{proof}
Let $\ln(R)$ to be the formal power series 
$\sum_{n=1}^{\infty} (-1)^{n+1}\frac{(R-1)^n}{n}$ (the Taylor series of $\ln$ at $1$).
It is the inverse function of $exp(R)$. 
Because $d_{\varnothing, \varnothing, \varnothing}(Y,D,E)=1$ for all possible $Y$, $D$, and $E$ and their Chern numbers form a full rank lattice in the vector space of all Chern numbers, the leading coefficient of  every $A_i$ is $1$. 
Therefore every $\ln A_i$ is a series of $y_{\beta}$ and $z_{\delta}$ without constant term.

Apply $\ln$ to two sides of Equation \eqref{eq:Ai} with $\alpha= \varnothing$ gives
$$ \ln (T_{\varnothing}(Y,D,E))
= \sum_{k=1}^m \int_Y\Theta_k(c_i(T_Y), c_1(D), c_i(E) ) \cdot \ln A_k.$$
 
The coefficient of $y_{\beta}z_{\delta}$ on the right hand side is the sum of $\int_Y\Theta_k$ multiplied by the coefficient of $y_{\beta}z_{\delta}$ in $\ln A_k$, which is linear in  $\int_Y\Theta_k$ and thus linear in the Chern numbers. 
The coefficient equals $\frac{a_{\beta, \delta} }{\#Aut(\beta)\cdot\#Aut(\delta)}$ therefore $a_{\beta, \delta}$ is also linear in the Chern numbers. 
\end{proof}

\begin{rem}
When there is no tangency conditions, for nodal curves on surfaces with at most eight nodes  \cite{KP} and nodal curves on $\PP^2$ with at most $14$ nodes \cite{BlockComputing}, the coefficients of $a_{\varnothing, \delta}$ are all integers. 
The general case is still open. 
\end{rem}

The existence of universal polynomials  can be easily deduced now. 
\begin{proof}[proof of Theorem \ref{thm:univ}] 
For a dimension $n-r-1$ ICIS $s$, let $m_s$ be the multiplicity of $s$ in $\alpha$. 
Equation \ref{eq:Ai} can be written as 
\begin{align}\label{ms} T_{\alpha}(Y,D,E)
= \prod_{k=1}^m A_k^{\int_Y\Theta_k(c_i(T_Y), c_1(D), c_i(E))}\cdot \prod_{s} P_{s}^{m_s}. \end{align}
where the second product is over all dimension $n-r-1$ ICIS. 

By Proposition \ref{prop:daL}, when $E$ is $(N(\alpha)+N(\beta)+N(\delta)+n-r+1)$-very ample the number of singular varieties with singular type $\delta$ satisfying tangency conditions $(\alpha, \beta)$ with $D$ is $d_{\alpha, \beta, \delta}(Y,D,E)$,  the coefficient of $y_{\beta}z_{\delta}$ in $T_{\alpha}(Y,D,E)$.
Because $\int_Y\Theta_k(c_i(T_Y), c_1(D), c_i(E))$ and $m_s$ are rational numbers, take binomial expansion of Equation \eqref{ms} then one can write  this coefficient  as a degree $|\beta|+|\delta|$ polynomial in Chern numbers and $m_s$, since $\int_Y\Theta_k(c_i(T_Y), c_1(D), c_i(E))$ is a linear combination of Chern numbers. 
\end{proof}

\begin{eg}\label{leading}

The leading terms of the universal polynomials can be computed as follows. 
Consider curves in the linear system of a line bundle $L$ on a surface $S$. 
From \cite{KP}, the universal polynomial of curves with a node ($\cA_1$-singularity) is 
$$T_{0,0,\cA_1} =3L^2+2LK_S+c_2(S).$$ 
The universal polynomial of curves with an ordinary triple point is
$$T_{0,0,\cD_4} =15L^2+20LK_S+5c_1(S)^2+5c_2(S).$$
By De Jonqui\`ere's Formula or using discriminant, if $S$ contains a line $D\cong \PP^1$, then the number of curves in $|L|$  tangent to $D$ is $2deg(L|_D)-2$. 
Then by Corollary \ref{exp}, the universal polynomial of curves in $|L|$ which has $a$ nodes and $b$ ordinary triple points and tangent to $D$ at $c$ points has leading term
$$(3L^2+2LK_S+c_2(S))^a (15L^2+20LK_S+5c_1(S)^2+5c_2(S))^b(2deg(L|_D)-2)^c/a!b!c!.$$
\end{eg}

The following theorem proves the polynomial property of the generalization of Caporaso-Harris invariants on smooth surfaces. The ampleness condition is weakened. 
\begin{theo}\label{surface} Denote  the tangency conditions of curves on surfaces by $\alpha = (\alpha_1, \alpha_2, \ldots)$ and $\beta = (0, \beta_2, \beta_3,\ldots)$ as in Caporaso-Harris invariants (see Section \ref{sect:non-induced}).  For fixed $r\in \NN$ and $\beta$, there is a degree $|\beta|+r$ polynomial $T_{\beta,r\cA_1}$ such that for  any smooth surface $S$, any smooth divisor $D$ and line bundle $L$ on $S$ and $\alpha$, the number of $r$-nodal curves in $|L|$ satisfying  tangency conditions $(\alpha,\beta)$ with $D$ is given by a polynomial $T_{\beta,r\cA_1}$ in sever Chern numbers $\fourtop, D^2, deg(L|_D), DK_S$ and $\alpha_i$ if $L$ is $(3r+\sum i(\alpha_i+\beta_i)+2)$-very ample. 
Furthermore, the generating series
$$T_{\alpha}(S,D,L) := \sum_{\beta, r} T_{\beta,r\cA_1}y_{\beta}z_{\cA_1}^r= A_1^{L^2}A_2^{LK_S}A_3^{c_1(S)^2}A_4^{c_2(S)}A_5^{D^2}A_6^{deg(L|_D)}A_7^{DK_S}\prod_{i}P_i^{\alpha_i}. $$
where $A_i\in \QQ[z_{\cA_1}]^{\times}$ for $1\leq i\leq 4$ and  $A_i\in \QQ[y_{\beta}, z_{\cA_1}]^{\times}$ for $5\leq i\leq 7$.  
\end{theo}

\begin{proof}
If one redefines $\xi(node) = \Spec \CC[x,y]/(x,y)^2$, curves on $S$ containing a subscheme isomorphic to $\xi(node)$ generically have node at the support of the subscheme.  
If the singularity is not node, then the curve must contain a subscheme isomorphic to  $\Spec \CC[x,y]/(y^2+ (x,y)^3)$ at the same point.

Run Section \ref{sect:deg} again with the new definitions will give the same degeneration formulae (Corollary \ref{goalpha}, \ref{d=dd2}). 
Therefore the universal polynomials and generating series still exist and satisfy same equations. 
But $L$ only has to be $(3r+\sum i(\alpha_i+\beta_i)+2)$-very ample since the length of $\xi(node)$ is $3$ and the invariants for tangency conditions are given in Example \ref{multiplicity}. 

There are seven Chern numbers of a line bundle $L$ and  a divisor $D$ on a surface $S$: 
$$\fourtop, D^2, deg(L|_D), DK_S.$$
Since the constant term of $T_{\varnothing}$ is $1$ and the Chern numbers of all possible $S$, sufficiently ample $L$, $D$ spans $\QQ^7$, $A_i\in \QQ[y_{\beta}, z_{\cA_1}]^{\times}$. 
If one lets $D$ be the zero divisor, then $\Lambda^{\Omega}_{\alpha,\beta, \delta}(Y,D,E)=0$ if $\alpha$ or $\beta$ is not empty. 
Therefore $A_i \in \QQ[z_{\cA_1}]^{\times}$ for $1\leq i\leq 4$.

\end{proof}

\begin{rem}

These $A_1$, $A_2$, $A_3$, $A_4$ are exactly   the $A_i$ in \cite[Theorem~1.3]{Tz}. 

The number of smooth curves with tangency conditions $(0, \beta)$ is given by the coefficient of $y_{\beta}$ in   $T_{\varnothing}(S,D,L)$. 
This is the same as counting the number of sections in $L|_D$ on $D$ with given vanishing multiplicities and is given by De Jonqui\`ere's Formula \cite[p.359]{ACGH}. 
Therefore  coefficients of $y_{\beta}$  in $A_5$, $A_6$ and $A_7$  are totally determined by De Jonqui\`ere's Formula for any $\beta$. 
\end{rem}

\section{Formulas}
The key idea of this article is expressing the number of singular varieties as  intersection numbers on Hilbert schemes. 
This approach can prove the existence of universal polynomials, but computing them from this point of view seems unfeasible because  the cycle $Y(\alpha,\beta, \delta)$  can be  very singular. 
The known formulae come from wide range of methods and usually only work in special situations. 
In the framework of this article, those results contribute to coefficients of generating series $T_{\alpha}(Y,D,E)$ or equivalently to universal polynomials.
If there are enough special cases to generate the algebraic cobordism group, all coefficients  can be computed. 

For general K3 surfaces and primitive line bundles, the generating series of nodal curves can be expressed in modular forms and quasi-modular forms \cite[Theorem~1.1]{Tz}. So it is natural to ask if other generating series (eg. on CY3 or log canonical manifold) will be in modular form and quasi-modular forms. 
Even if not, then is it possible to find a closed form of generating series? 
So far we can not answer these questions. 
More computations will be helpful to support or disprove such conjectures.

The currently known formulae are listed below.
In Example \ref{leading}, we compute leading terms of universal polynomials from the case of one singular point or one tangency point. 
The same method works in general by Corollary \ref{exp} so the formulae below do not only apply to their original setting but also contribute to infinitely many  universal polynomials. 


\subsection{Nodal curves on surfaces}

\begin{itemize}

\item Caporaso-Harris' recursive formula \cite{CH} determines the numbers of plane nodal curves with arbitrary tangency conditions with a given line without any ampleness assumption. 

\item Vakil \cite{Vakil} proved a similar recursive formula for nodal curves on Hirzebruch surfaces. $T_{\alpha, \beta, r\cA_1}$ can be determined from Caporaso-Harris' and Vakil's formulas. 

\item The universal polynomial of $r$-nodal curves on surfaces can be  determined by combining any two of  \cite{BL}, \cite{CH} and \cite{Vakil}, or by \cite{KST}. The explicit formulas for $r\leq 6 $ and $r \leq 8$ were found by Vainsencher and Kleiman-Piene  (\cite{Va}, \cite{KP}).
Block computed the equivalent form as Corollary \ref{exp} for $r\leq 14$ \cite{BlockComputing}. 
Kool, Shende and Thomas \cite{KST} provided an algorithm for any $r$. 

\item Block \cite{Blockrel} used tropical geometry to prove the Caporaso-Harris invariants are polynomials when degree of the curves are relatively large. He computed $T_{\alpha, \beta, r\cA_1}$ for $(\PP^2, \cO(d))$ when $1 \leq r\leq  6$ and any $\alpha$, $\beta$. 

\item In \cite{TzComputing}, we computed the first dozens of coefficients for $A_i$ and $P_i$ in Theorem \ref{surface}. Some properties are proven and conjectured there. 
\end{itemize}

\subsection{Curves with arbitrary singularity on surfaces}

\begin{itemize}
\item Kleiman and Piene \cite{KP}  computed the universal polynomials for many singularities of low codimensions 
such as $\delta=(D_4)$, $(D_4, A_1)$, $(D_4, A_1, A_1)$, $(D_4, A_1, A_1, A_1)$, $(D_6)$, $(D_6, A_1)$ and  $(E_7)$. 

\item When  there is only one singular point, Kerner \cite{Ke} found an algorithm to enumerate the number of plane curves with one fixed topological type singularity, provided that the normal form is known. 

\item If there is no singularity condition, then it is equivalent to  finding sections of line bundles on smooth curves which vanish at points of given  multiplicities. The answer is given by De Jonqui\`ere's Formula \cite[p.359]{ACGH}.  

\item Recently Basu and Mukherjee \cite{BM1, BM2} used topological methods to compute the number of plane curves with one or two singularities when the first singularity  is $ADE$-singularity or ordinary fourfold points and the additional one is a node. 
\end{itemize}

\subsection{General Dimension}

\begin{itemize}
\item Vainsencher \cite{Vain_six} computed the number hypersurfaces with less or equal to six double points. 

\item In a series of paper  (\cite{Kaz_thom, Kaz} and a lot more on line),  M. Kazarian showed Legendre characteristic classes  should give universal polynomials for singular hypersurfaces, and they are independent of the dimension! The explicit formulas of Legendre characteristic classes when the codimension is less or equal to six are listed in \cite{Kaz_thom}. 
They agree with Kleiman and Piene's results on surfaces. The transversality conditions and the domain of universality are quite subtle. 
\end{itemize}

\bibliography{mybib}
\bibliographystyle{amsplain}
\end{document}